\documentclass[12pt,reqno]{amsart}
\usepackage{xcolor}
\usepackage{amsfonts,color}
\usepackage{latexsym,amssymb}
\usepackage{amsmath}
\usepackage[utf8]{inputenc}
\usepackage{dsfont}

\usepackage{graphicx} % Required for inserting images
\usepackage{tikz, tkz-euclide}
\usetikzlibrary{patterns,patterns,patterns.meta}

\newtheorem{theorem}{Theorem}[section]
\newtheorem{lemma}[theorem]{Lemma}

\newtheorem{corollary}[theorem]{Corollary}
\newtheorem{remark}[theorem]{Remark}

\usepackage[left=2.7cm,right=2.7cm,top=3.2cm,bottom=3.2cm]{geometry}

\usepackage[bookmarksnumbered,colorlinks, linkcolor=blue, citecolor=red, pagebackref, bookmarks, breaklinks]{hyperref}

\usepackage[hyperpageref]{backref}

\catcode`@=11 \@addtoreset{equation}{section} \catcode`@=12

\begin{document}
	\title[Liouville theorems and existence for quasilinear elliptic problems]{Liouville-type theorems and existence of solutions for quasilinear elliptic problems}
	
	\author[J. M.\ do \'O]{J. M. do \'O}
	\address[J. M.\ do \'O]{Department of Mathematics,
		Federal University of Para\'{\i}ba
		\newline\indent
		58051-900, Jo\~ao Pessoa-PB, Brazil}
	\email{\href{mailto:jmbo@mat.ufpb.br}{jmbo@mat.ufpb.br}}
	
	\author[Freire]{R. F. Freire}
	\address[R. F. Freire]{Department of Mathematics,
		Federal University of Para\'{\i}ba
		\newline\indent
		58051-900, Jo\~ao Pessoa-PB, Brazil}
	\email{\href{mailto:ranieri.franca@academico.ufpb.br}{ranieri.franca@academico.ufpb.br}}
	
	\author[Medeiros]{E. S. Medeiros}
	\address[E. S. Medeiros]{Department of Mathematics,
		Federal University of Para\'{\i}ba
		\newline\indent
		58051-900, Jo\~ao Pessoa-PB, Brazil}
	\email{\href{mailto:everaldo@mat.ufpb.br}{everaldo@mat.ufpb.br}}

	%\date{\today}
	%\hyphenation{}

\begin{abstract} 
This study establishes Liouville-type theorems for indefinite quasilinear elliptic equations in the upper half-space. Additionally, we demonstrate the existence of solutions for this class of problems using the fibering method. Our approach relies on a novel weighted Sobolev embedding developed for the upper half-space.
\end{abstract}

\maketitle

%%%%%%%%%%%%%%%%%%%%%%%%%%%%%%%%%%%%%%%%%%%%%%%%%%%%%%%%%%%%%%%%%%%%%%%%%%%%%%%%%%%%%%%%%%%%%%%%%%
	%  TABLE OF CONTENTS
%%%%%%%%%%%%%%%%%%%%%%%%%%%%%%%%%%%%%%%%%%%%%%%%%%%%%%%%%%%%%%%%%%%%%%%%%%%%%%%%%%%%%%%%%%%%%%%%%%
	
% \begin{center}
% %\begin{minipage}{8cm}
% \footnotesize
% \tableofcontents
% %\end{minipage}
% \end{center}
	
 %%%%%%%%%%%%%%%%%%%%%%%%%%%%%%%%%%%%%%%%%%%%%%%%%%%%%%%%%%%%%%%%%%%%%%%%%%%%%%%%%%%%%%%%%%%%%%%%%%
	% SECTION 1 %%%%%%%%%%%%%%%%%%%%%%%%%%%%%%%%%%%%%%%%%%%%%%%%%%%%%%%%%%%%%%%%%%%%%%%%%%%%%%%%%%%%%
 %%%%%%%%%%%%%%%%%%%%%%%%%%%%%%%%%%%%%%%%%%%%%%%%%%%%%%%%%%%%%%%%%%%%%%%%%%%%%%%%%%%%%%%%%%%%%%%%%%
\section{Introduction}
 
 %%%%%%%%%%%%%%%%%%%%%%%%%%%%%%%%%%%%%%%%%%%%%%%%%%%%%%%%%%%%%%%%%%%%%%
 % Our problem
 %%%%%%%%%%%%%%%%%%%%%%%%%%%%%%%%%%%%%%%%%%%%%%%%%%%%%%%%%%%%%%%%%%%%%%
This work addresses various aspects of quasilinear elliptic problems with Neumann boundary conditions. The discussion focuses on Liouville-type results and the existence of solutions for the following model of quasilinear elliptic problems:
\begin{equation}\label{P}
		\left\lbrace
		\begin{aligned}
			-\text{div}(\rho(x)|\nabla u|^{p-2}\nabla u ) &=a(x)|u|^{q-2}u-b(x)|u|^{s-2}u,\quad&\mathds{R}^N_+&,
			\vspace{0.2cm}\\
			|\nabla u|^{p-2}\nabla u\cdot\nu&=0,\quad &\mathds{R}^{N-1},
		\end{aligned}
		\right.
		\tag{$\mathcal{P}$}
	\end{equation} 
where $\mathds{R}^N_+=\{(x',x_N)\in \mathds{R}^N: x' \in \mathds{R}^{N-1}, x_N>0\}$ standards for the upper half-space,  $\nu$ is the unit outer normal to the boundary, $\partial\mathds{R}^N_+:=\mathds{R}^{N-1}$, $1< q, s \leq p^*$ if $1<p<N$ and $1< q, s<\infty$ if $p=N$.

Here and throughout this paper, we assume that $\rho, \, a, \, b  \in L^1_{\mathrm{loc}}(\mathds{R}^N_+)$  and are positive functions.
For $1<p<N$, we denote by $p^*:=Np/{(N-p)}$ the critical exponent for the Sobolev embedding and  $p^*=\infty$ if $p=N$. 

 %%%%%%%%%%%%%%%%%%%%%%%%%%%%%%%%%%%%%%%%%%%%%%%%%%%%%%%%%%%%%%%%%%%%%%
 % Related results
 %%%%%%%%%%%%%%%%%%%%%%%%%%%%%%%%%%%%%%%%%%%%%%%%%%%%%%%%%%%%%%%%%%%%%%
From a mathematical perspective, the nature of \eqref{P} is described according to the behavior of the competing terms $a(x)|u|^{q-2}u$ and $b(x)|u|^{s-2}u$ as determined by the integrability properties of the ratio $a(x)^{1/p}/b(x)^{1/s}$ (as discussed in  Alama-Tarantello  \cite{zbMATH00484529,zbMATH00943136}). The interplay between the weight functions $a(x)$ and $b(x)$ significantly impacts the existence and nonexistence of solutions to \eqref{P} and has garnered substantial attention among researchers, see, for instance, \cite{zbMATH00822783,zbMATH05308896}. 
We mention that the weight functions are not necessarily spherically symmetric.
Thus, we are motivated to pursue new weighted Sobolev embeddings to enable variational frameworks in diverse settings.

In the works \cite{zbMATH05077864,zbMATH05640324}, based on a Hardy-type inequality due to K. Pfl\"{u}ger, \cite{zbMATH01148489} (see also \cite{zbMATH04019575}) and the fibering method,  it was established the existence and Liouville-type results for a similar class of quasilinear elliptic problems with Robin boundary condition in an unbounded domain $\Omega\subset \mathds{R}^N$ with noncompact smooth boundary, $1<p<N$, $q,s \in (1,p^*)$, $\rho \in L^\infty(\Omega)\cap L^\infty(\partial \Omega) $ and $0<\rho_0<\rho(x) $,  where the potentials $a$ and $b$ vanish at infinity. 

It is essential to mention that the approach in \cite{zbMATH05640324} used to treat a problem with Robin boundary conditions cannot be used to study problems with Neumann boundary conditions because their argument is based on K. Pfl\"{u}ger's inequality, which does not allow one to eliminate the boundary term. We also highlight that based on a Hardy-type inequality in \cite{EMEDOOEVE2010}, the existence and nonexistence results for semilinear elliptic problems with Robin boundary conditions were addressed using a variational approach. For related results, see also \cite{zbMATH05308896}.

Our approach here is based on a new class of Hardy-type inequalities, which allows us to consider problems with Neumann boundary conditions. We also emphasize that we have determined the associated constants for these inequalities. Similar to the classical Hardy and Sobolev inequalities in $\mathds{R}^N$, we believe that we have obtained the optimal values of the associated constants, in contrast to the results in \cite{zbMATH05640324}, where the exact constants are unknown. Hence, we gave a partial answer to a question raised in \cite{zbMATH05640324}. 
With these results, we obtained more precise a priori estimates for eventual solutions of \eqref{P} to obtain Liouville-type results.  
Moreover, we have incorporated the extreme scenario where $p=N$ into our analysis.

%%%%%%%%%%%%%%%%%%%%%%%%%%%%%%%%%%%%%%%%%%%%%%%%%%%%%%%%%%%%%%%%%%%%%%
 % Assumptions
 %%%%%%%%%%%%%%%%%%%%%%%%%%%%%%%%%%%%%%%%%%%%%%%%%%%%%%%%%%%%%%%%%%%%%%

\subsection{Assumptions} 
Henceforth, we presume that the weight function $\rho$ adheres to the following technical hypothesis:
	\begin{itemize}
		\item [$(H_0)$] there are constants $\rho_0>0$ and  $\gamma>p-1$ such that
		$$
		\rho(x)\geq \rho_0(1+x_N)^\gamma\quad \text{a.e. in $\mathds{R}^N_+$}.
		$$
		\end{itemize}

The assumption $(H_0)$ will be crucial in our argument based on a class of Hardy-type inequality, which will be stated in the following section. We also mention that $(H_0)$ prevent us from considering equations involving the pure $p-$Laplace operator, based in the Beppo-Levi space $D^{1,p}$.
 
First, we must introduce our variational setting to describe our results for \eqref{P}. Let $C_\delta^\infty(\mathds{R}^N_+)$ be the set of all functions $u \in C_0^\infty(\mathds{R}^N)$ restricted to $\mathds{R}^N_+$ and consider the weighted space $E$ defined as the closure of $C_\delta^\infty(\mathds{R}^N_+)$ with respect to the norm
	$$
		\|u\|:=\left(\int_{\mathds{R}^{N}_+}\rho(x)|\nabla u|^p\, \mathrm{d} x\right)^{1/p}.
	$$
 
	Here, by a weak solution of \eqref{P} we mean a function $u \in E$ such that  
	\begin{equation}\label{weak solution}
		\int_{\mathds{R}^{N}_+}\rho(x)|\nabla u|^{p-2}\nabla u \nabla \varphi \, \mathrm{d} x=\int_{\mathds{R}^{N}_+}(a(x)|u|^{q-2}u  -b(x)|u|^{s-2}u) \varphi \, \mathrm{d} x, 
	\end{equation}
	holds for every $\varphi \in C_\delta^\infty(\mathds{R}^N_+)$.
\bigskip 
 
To state our  nonexistence results, we shall introduce the following class of functions:
\begin{equation*}
 \mathcal{K}_b:=\left\{k\in C(\overline{\mathds{R}^N_+},(0,\infty)) : k(x)(1+x_N)^{p-\gamma}  \in L^{\infty}(\mathds{R}^{N}_+) \right\}.   \end{equation*}
 
\bigskip 

%%%%%%%%%%%%%%%%%%%%%%%%%%%%%%%%%%%%%%%%%%%%%%%%%%%%%%%%%%%%%%%%%%%%%%
 % Description of the main results
%%%%%%%%%%%%%%%%%%%%%%%%%%%%%%%%%%%%%%%%%%%%%%%%%%%%%%%%%%%%%%%%%%%%%%
\section{Description of the main results}

\subsection{Liouville-type results}
Let us introduce some notation that will be used throughout the paper. We denote
$$
			C_{p,\gamma}:=\frac{\gamma-p+1}{p}
$$
  and
  \begin{equation}\label{GeneralConstant}
  \eta (r,q,t):=\frac{(t-r)^{t-r}}{(q-r)^{q-r}(t-q)^{t-q}}\quad\text{if}\quad r<q<t.
  \end{equation}
  
Our first concern is to declare nonexistence when $s<q<p$.

\begin{theorem}[p-sublinear case]\label{nonexistence1} 
 Assume $(H_0)$ and suppose that $b\in \mathcal{K}_b\cap L^1(\mathds{R}^N_+)$. If $1< s<q<p\leq N$ and $a/b\in L^\infty(\mathds{R}^N_+)$ with 
		\begin{equation}\label{condition for nonexistence}
			\left\|\frac{a}{b}\right\|_{\infty}^{p-s}	\left(\frac{b_0C_{p,\gamma}^{-p}}{\rho_0}\right)^{q-s}<\eta(s,q,p),
		\end{equation}
then \eqref{P} possesses only the trivial weak solution. Hereafter, $b_0>0$ denotes a constant such that 
$$
b(x)(1+x_N)^{p-\gamma}\leq b_0\quad\mbox{in}\quad\mathds{R}^N_+.
$$
\end{theorem}

\begin{corollary} If $1< s<q<p\leq N$, the quasilinear elliptic problem 	\begin{equation}\label{P}
		\left\lbrace
		\begin{aligned}
			-\mathrm{div}(\rho(x)|\nabla u|^{p-2}\nabla u ) &=f(x,u),\; &\mathds{R}^N_+&,
			\vspace{0.2cm}\\
			|\nabla u|^{p-2}\nabla u\cdot\nu&=0,\; &\mathds{R}^{N-1},
		\end{aligned}
		\right.
		\tag{$\mathcal{P}$}
	\end{equation} 
with 
\begin{equation*}
f(x,u)=    \lambda \, \frac{(1+x_N)^{\gamma-p}}{(1+|x|)^{\theta_1}}|u|^{q-2}u-\frac{(1+x_N)^{\gamma-p}}{(1+|x|)^{\theta_2}}|u|^{s-2}u
\end{equation*}
possesses only the trivial weak solution whenever $\lambda>0$ is sufficiently small and
\begin{equation*}
    \max\{N,N+\gamma-p\}<\theta_2\leq \theta_1 .
\end{equation*}
\end{corollary}

In our second nonexistence result, we address the case where 
$p<q<s$.

 \begin{theorem}[p-superlinear case]\label{nonexistence2} 
Assume $(H_0)$ and suppose that $b\in \mathcal{K}_b$. If $1<p<N$, $p<q<s\leq p^*$ and $a/b\in L^\infty(\mathds{R}^N_+)$ with
 \begin{equation}\label{condition for nonexistence 2}
			 \left\|\frac{a}{b}\right\|_{\infty}^{s-p}	\left(\frac{b_0C_{p,\gamma}^{-p}}{\rho_0}\right)^{s-q}<\eta(p,q,s),
		\end{equation}
then \eqref{P} possesses only the trivial weak solution. Moreover, the same result holds if $p=N$ and $p<q<s<\infty$. 
	\end{theorem}

\bigskip

\begin{corollary} If $1< s<q<p\leq N$, 
\begin{equation*}
f(x,u)=    \lambda \, (1+x_N)^{\theta_1}|u|^{q-2}u
-\mu \, (1+x_N)^{\theta_2}|u|^{s-2}u
\end{equation*}
and 
\begin{equation*}
    \theta_1\leq \theta_2\leq \gamma-p
\end{equation*}
the quasilinear elliptic problem 
\begin{equation}\label{P}
		\left\lbrace
		\begin{aligned}
			-\mathrm{div}(\rho(x)|\nabla u|^{p-2}\nabla u ) &=f(x,u),\; &\mathds{R}^N_+&,
			\vspace{0.2cm}\\
			|\nabla u|^{p-2}\nabla u\cdot\nu&=0,\; &\mathds{R}^{N-1},
		\end{aligned}
		\right.
		\tag{$\mathcal{P}$}
	\end{equation} 
possesses only the trivial weak solution whenever $\lambda>0$ is sufficiently small or $\mu>0$ sufficiently large.
\end{corollary}

\bigskip

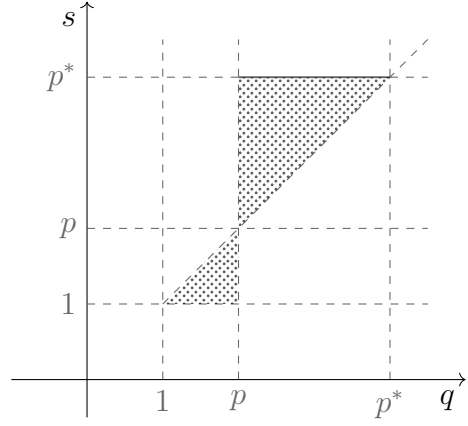
\begin{figure}[h]
    \centering
\begin{tikzpicture}[ ]
% eixos coordenados
    \draw[->] (-1,0) -- (5,0) coordinate (q) node[below left] {$q$};
    \draw[->] (0,-0.5) -- (0,5) coordinate (s) node[below left] {$s$};
% retas horizontais
    \draw[-,dashed,darkgray, opacity=0.8] (0,1) node[left] {$1$} -- (4.5,1);
    \draw[-,dashed,darkgray, opacity=0.8] (0,2) node[left] {$p$} -- (4.5,2);
    \draw[-,dashed,darkgray, opacity=0.8] (0,4) node[left] {$p^\ast$} -- (2,4);
    \draw[-,line width=0.25mm,darkgray, opacity=0.8] (2,4) -- (4,4);
    \draw[-,dashed,darkgray, opacity=0.8] (4,4) -- (4.5,4);
%retas verticais
    \draw[-,dashed,darkgray, opacity=0.8] (1,0) node[below] {$1$} -- (1,4.5);
    \draw[-,dashed,darkgray, opacity=0.8] (2,0) node[below] {$p$} -- (2,4.5);
    \draw[-,dashed,darkgray, opacity=0.8] (4,0) node[below] {$p^\ast$} -- (4,4.5);
%reta transversal
    \draw[-,dashed,darkgray, opacity=0.8] (1,1) -- (4.5,4.5);
%regioes
    \fill[pattern=crosshatch dots, opacity=0.8] (1,1) -- (2,2) -- (2,1) -- cycle;
    \fill[pattern=crosshatch dots, opacity=0.8] (2,4) -- (2,2) -- (4,4) -- cycle;
\end{tikzpicture}
    \caption{Nonexistence of solutions}
    \label{fig:enter-label}
\end{figure}

The basic idea to prove Theorems~\ref{nonexistence1} and \ref{nonexistence2} relies on refining the arguments presented in \cite{zbMATH05640324} by using a specific a priori estimate and a new Hardy-type inequality derived in Subsection~\ref{HS}.

\subsection{Existence results}
To establish our existence results, it is necessary to impose additional hypotheses on the weight functions $a$ and $b$ to ensure the compactness of the Sobolev embedding, thereby enabling the application of the fibering method as demonstrated in papers \cite{zbMATH05077864,zbMATH05640324}. 
To this end, we introduce the following class of functions:
\begin{equation*}
 \mathcal{K}_0:=\left\{k\in C(\overline{\mathds{R}^N_+},(0,\infty))\;\mbox{such that}\;\lim_{|x|\rightarrow\infty}k(x)(1+x_N)^{p-\gamma}=0\right\}.   
 \end{equation*}

Next, we assume $1 < p \leq N$ to state our existence results.  Our first result considers $p-$superlinear case $p < s < q $ or when  $ s < p < q $.

\begin{theorem}\label{Primeiro Existencia} If $(H_0)$ holds, then \eqref{P} has a nontrivial and nonnegative weak solution when one of the following conditions occurs:
\begin{itemize}
    \item[i)] $p < s < q < p^*$ and $a, b \in \mathcal{K}_0$;
    \item[ii)] $1 < s < p < q < p^*$, $a \in \mathcal{K}_0$, and $b \in L^1(\mathds{R}^N_+) \cap \mathcal{K}_b$.
\end{itemize}
\end{theorem}

On our second existence result, we treat the $p-$sublinear case $  s < q <p $ or when  $ q < p < s $.

\begin{theorem}\label{Segundo Existencia} If $(H_0)$ holds,  then \eqref{P} has a nontrivial and nonnegative weak solution when one of the following conditions occurs:
\begin{itemize}
\item[i)] $1<q<s<p$ and $a,b \in L^1(\mathds{R}^N_+)\cap \mathcal{K}_b$;
\item[ii)] $1<q<p<s<p^*$, $a\in L^1(\mathds{R}^N_+)\cap \mathcal{K}_b$ and $b \in \mathcal{K}_0$.
\end{itemize}
\end{theorem}

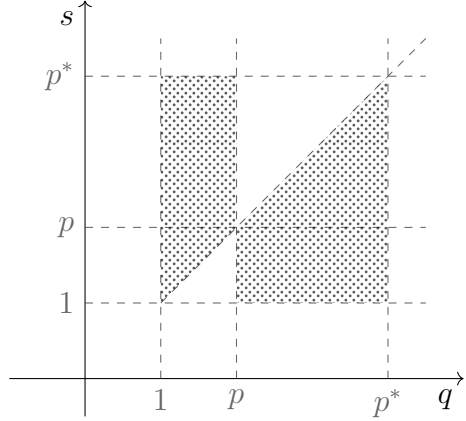
\begin{figure}[h]
    \centering
\begin{tikzpicture}[ ]
% eixos coordenados
    \draw[->] (-1,0) -- (5,0) coordinate (q) node[below left] {$q$};
    \draw[->] (0,-0.5) -- (0,5) coordinate (s) node[below left] {$s$};
% retas horizontais
    \draw[-,dashed,darkgray, opacity=0.8] (0,1) node[left] {$1$} -- (4.5,1);
    \draw[-,dashed,darkgray, opacity=0.8] (0,2) node[left] {$p$} -- (4.5,2);
    \draw[-,dashed,darkgray, opacity=0.8] (0,4) node[left] {$p^\ast$} -- (4.5,4);
%retas verticais
    \draw[-,dashed,darkgray, opacity=0.8] (1,0) node[below] {$1$} -- (1,4.5);
    \draw[-,dashed,darkgray, opacity=0.8] (2,0) node[below] {$p$} -- (2,4.5);
    \draw[-,dashed,darkgray, opacity=0.8] (4,0) node[below] {$p^\ast$} -- (4,4.5);
%reta transversal
    \draw[-,dashed,darkgray, opacity=0.8] (1,1) -- (4.5,4.5);
%regioes
    \fill[pattern=crosshatch dots, opacity=0.8] (1,1) -- (1,4) -- (2,4) -- (2,2) -- cycle;
    \fill[pattern=crosshatch dots, opacity=0.8] (2,1) -- (4,1) -- (4,4) -- (2,2) -- cycle;
\end{tikzpicture}
    \caption{Existence of nontrivial solutions}
    \label{fig:enter-label}
\end{figure}

To present our third existence result, we consider the functionals defined on $E$ by

\begin{equation}\label{FuntionalAB}
A(u)=\int_{\mathds{R}^N_+}a(x)|u|^q\, \mathrm{d} x\quad\text{and}\quad B(u)=\int_{\mathds{R}^N_+}b(x)|u|^s\, \mathrm{d} x.
 \end{equation}
We note that under the assumptions of Theorem~\ref{nonexistence2}, for $v\in E$, the following key inequality holds true  (see Lemma~\ref{Estimativa 2})
$$
A(v)^{s-p}<\eta(p,q,s)B(v)^{q-p}\|v\|^{p(s-q)}<\eta(p,q,s)\left(\frac{q}{p}\right)^{s-q} B(v)^{q-p}\|v\|^{p(s-q)}.
$$
We consider the existence of solutions for \eqref{P} in a subset of the complementary case of this inequality. 
Precisely, considering the set 
\begin{equation}\label{D hypothesis}
\mathcal{D}_1:= \left\{u \in E: A(u)^{s-p}>\left(\frac{q}{p}\right)^{s-q}\eta(p,q,s)B(u)^{q-p}\|u\|^{p(s-q)}\right\},
\end{equation}
a counterpart of Theorem \ref{nonexistence2} reads as follows.

\begin{theorem}\label{Terceiro Exitencia} 
Assume that $(H_0)$ holds. If $a,b\in \mathcal{K}_0, \; p<q<s<p^*, \; \mathcal{D}_1 \neq\emptyset$ and
\begin{equation}\label{ab condition}
\frac{a^{1/q}}{b^{1/s}}\in L^{\frac{sq}{s-q}}(\mathds{R}_+^N),
\end{equation}
then, \eqref{P} has a nontrivial and nonnegative weak solution.
\end{theorem}

\begin{remark}
The functions $a=b=\lambda k$, with $k$ given by
\[
k(x)=(1+x_N)^{\gamma-p}(1+|x|)^{-\theta},
\]
satisfy the assumptions of Theorem \ref{Terceiro Exitencia} for $\theta>\max\{N,N+\gamma-p\}$ and $\lambda$ sufficiently large. First, we observe that 
\[
\left[\frac{a(x)^{1/q}}{b(x)^{1/s}}\right]^{\frac{sq}{s-q}}=a(x)=\lambda (1+x_N)^{\gamma-p}(1+|x|)^{-\theta}\in L^1(\mathds{R}^N_+), 
\]
whenever $\theta>N+\gamma-p$. For $u \in E\backslash\{0\}$ fixed, one has 
$$
{A(u)^{s-p}}/{B(u)^{q-p}}=\lambda^{s-q}\|u\|_{L^q(\mathds{R}^N_+,k)}^{q(s-p)}\|u\|_{L^s(\mathds{R}^N_+,k)}^{s(p-q)}.
$$
Since $s>q$, for $\lambda$ sufficiently large we see that $$
A(u)^{s-p}>\left(\frac{q}{p}\right)^{s-q}\eta(p,q,s)B(u)^{q-p}\|u\|^{p(s-q)},
$$ 
and hence $\mathcal{D}_1\neq \emptyset$.
\end{remark}

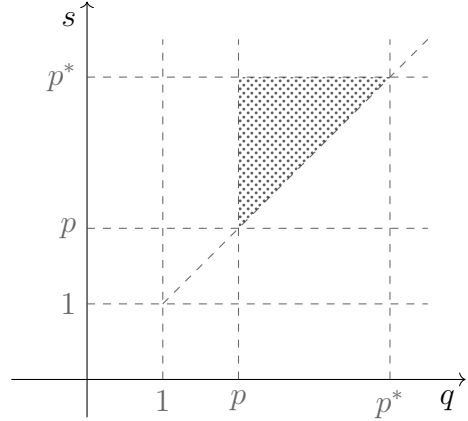
\begin{figure}[h]
    \centering
\begin{tikzpicture}[ ]
% eixos coordenados
    \draw[->] (-1,0) -- (5,0) coordinate (q) node[below left] {$q$};
    \draw[->] (0,-0.5) -- (0,5) coordinate (s) node[below left] {$s$};
% retas horizontais
    \draw[-,dashed,darkgray, opacity=0.8] (0,1) node[left] {$1$} -- (4.5,1);
    \draw[-,dashed,darkgray, opacity=0.8] (0,2) node[left] {$p$} -- (4.5,2);
    \draw[-,dashed,darkgray, opacity=0.8] (0,4) node[left] {$p^\ast$} -- (4.5,4);
%retas verticais
    \draw[-,dashed,darkgray, opacity=0.8] (1,0) node[below] {$1$} -- (1,4.5);
    \draw[-,dashed,darkgray, opacity=0.8] (2,0) node[below] {$p$} -- (2,4.5);
    \draw[-,dashed,darkgray, opacity=0.8] (4,0) node[below] {$p^\ast$} -- (4,4.5);
%reta transversal
    \draw[-,dashed,darkgray, opacity=0.8] (1,1) -- (4.5,4.5);
%regioes
    %\fill[pattern=crosshatch dots, opacity=0.8] (1,1) -- (2,2) -- (2,1) -- cycle;
    \fill[pattern=crosshatch dots, opacity=0.8] (2,4) -- (2,2) -- (4,4) -- cycle;
\end{tikzpicture}
    \caption{Existence of nontrivial solutions}
    \label{fig:enter-label}
\end{figure}

\begin{remark}
    The proofs of Theorems \ref{Primeiro Existencia}, \ref{Segundo Existencia} and \ref{Terceiro Exitencia}  are based on the classical fibering method, see \cite{zbMATH01060125,zbMATH05077864,zbMATH05640324,zbMATH03714430, zbMATH05635193}.
   It is worth mentioning that if $q<s=p$, the Direct Methods in the Calculus of Variations ensure the existence of solutions to \eqref{P}. In the case that $s=p<q$, the mountain-pass approach can be applied to establish the existence of solutions to \eqref{P}.
\end{remark}

\subsection{Hardy-Sobolev type inequalities}\label{HS}
Important ingredients in our analysis are some new weighted Sobolev embeddings. We begin by stating a Hardy-Sobolev type inequality, which will play a central role in proving our main results. Specifically, we aim to demonstrate the following:

\begin{theorem}[Hardy]\label{Hardy}
		Let $N\geq2$ and $\gamma>p-1>0$. Then, for every $u \in C^{\infty}_0(\mathds{R}^N)$ it holds
		\begin{equation}\label{hardy with xn}
			 C_{p,\gamma}^p \int_{\mathds{R}_{+}^N}\frac{|u|^p}{(1+x_N)^{p-\gamma}}\, \mathrm{d}  x  + 
			C_{p,\gamma}^{p-1} \int_{\mathds{R}^{N-1}}|u|^p \, \mathrm{d}  x'
   \leq \int_{\mathds{R}_{+}^N}(1+x_N)^\gamma |\nabla u|^p \, \mathrm{d} x .
		\end{equation}		
	\end{theorem}
For related Hardy-type inequalities on the upper half-space, we refer to the works \cite{zbMATH05115299, zbMATH06347299,zbMATH06053053, zbMATH05134242,zbMATH03073200,zbMATH02172983}.  In the mentioned context, we are considering functions in $C_0^\infty(\mathds{R}^N)$,
 while in the aforementioned work, the inequalities are established for functions in $C_0^\infty(\mathds{R}^N_+)$, as seen, for instance, in \cite[Theorem 6.9]{zbMATH06347299}.\\

 In our next findings, we shall explore weighted Sobolev inequalities. To this end, given  $\gamma >p-1>0$, we consider the Sobolev space  $\mathcal{D}_\gamma^{1, p}(\mathds{R}^N_+)$ defined as the completion of the space $C_\delta^\infty(\mathds{R}^N_+)$ with respect to the norm  
	$$
	\|u\|_{\mathcal{D}_\gamma^{1, p}(\mathds{R}^N_+)}=\left(\int_{\mathds{R}_+^N}(1+x_N)^{\gamma}|\nabla u|^p\, \mathrm{d} x+\int_{\mathds{R}^N_+}\frac{|u|^p}{(1+x_N)^{p-\gamma}}\, \mathrm{d} x\right)^{1/p}.
	$$

From the definition of $\mathcal{D}_\gamma^{1, p}(\mathds{R}^N_+)$ and a density argument we see that  
\begin{equation}\label{hardy Fraca}
			 C_{p,\gamma}^p \int_{\mathds{R}_{+}^N}\frac{|u|^p}{(1+x_N)^{p-\gamma}}\, \mathrm{d}  x  
   \leq \int_{\mathds{R}_{+}^N}(1+x_N)^\gamma |\nabla u|^p \, \mathrm{d} x , \quad \forall u\in \mathcal{D}_\gamma^{1, p}(\mathds{R}^N_+).
		\end{equation}
Together with Theorem~\ref{Hardy}, this implies the following crucial result for our purpose.
\begin{corollary} For $\gamma>p-1>0$, in the space $\mathcal{D}_\gamma^{1, p}(\mathds{R}^N_+)$ the norm $\|\cdot\|_{\mathcal{D}_\gamma^{1, p}(\mathds{R}^N_+)}$ is equivalent to
		$$
		\|u\|_\gamma:=\left(\int_{\mathds{R}^N_+}(1+x_N)^\gamma|\nabla u|^p\, \mathrm{d} x\right)^{1/p}.
		$$
  In particular, from $(H_0)$ we have the continuous embedding 
  $$
  \left(E,\|\cdot\|\right)\hookrightarrow \left(\mathcal{D}_\gamma^{1, p}(\mathds{R}^N_+\right),\|\cdot\|_\gamma)\hookrightarrow (\mathcal{D}_\gamma^{1, p}(\mathds{R}^N_+),\|\cdot\|_{\mathcal{D}_\gamma^{1, p}(\mathds{R}^N_+)}).
  $$
	\end{corollary}
	
	\begin{theorem}[Sobolev inequality]\label{Sobolev} Let $N\ge2$ and $\gamma>p-1>0$. Then, there exists a constant $C_0>0$ such that  
		$$
			\left(\int_{\mathds{R}^{N}_+}\frac{|u|^q}{(1+x_N)^{p-\gamma}}\, \mathrm{d} x\right)^{p/q} \leq C_0 \int_{\mathds{R}^{N}_+}(1+x_N)^{\gamma}|\nabla u|^p\, \mathrm{d} x,\quad \forall u \in \mathcal{D}_\gamma^{1, p}(\mathds{R}^N_+)
		$$
		whenever $q\in[p,p^*]$ if $1<p<N$ and $q\in[p,\infty)$ if $N=p$.
	\end{theorem}
\begin{remark} We finally highlight that our results for \eqref{P} can be extended to a more general class of problems of the form
$$
			\left\lbrace
			\begin{aligned}
				-{\rm{div}}(\rho(x)|\nabla u|^{p-2}\nabla u ) +\lambda\frac{|u|^{p-2}u}{(1+x_N)^{p-\gamma}}&=a(x)|u|^{q-2}u-b(x)|u|^{s-2}u&\mbox{in }&\,\mathds{R}^N_+,
				\vspace{0.2cm}\\
				\rho(x)|\nabla u|^{p-2}\nabla u\cdot\nu+\mu |u|^{p-2}u&=0,&\mbox{on }&\,\mathds{R}^{N-1}.
			\end{aligned}
			\right.
$$
Considering the norm 
$$
		\|u\|^p_{\lambda, \mu} = \int_{\mathds{R}^{N}_+}\rho(x)|\nabla u|^p\, \mathrm{d} x+\lambda\int_{\mathds{R}^{N}_+}\frac{|u|^p}{(1+x_N)^{p-\gamma}}\, \mathrm{d} x+\mu \int_{\mathds{R}^{N-1}}|u|^p\, \mathrm{d} x',
  $$
and invoking Theorem \ref{Hardy}, we can see that, for certain conditions on the parameters $\lambda,\mu$  depending on $\rho_0$ and $C_{p,\gamma}$, the norms $\|\cdot\|_{\lambda,\mu}$ is equivalent to $\|\cdot\|$.
Therefore, the same approach can treat this more general class of problems. 
\end{remark}

\paragraph{Strategy of the proofs}
First, we obtain a Hardy-type inequality in the upper half-space for functions in $C_0^\infty(\mathds{R}^N)$, Theorem \ref{Hardy}.
Consequently, in combination with interpolation and interaction argument, we establish a weighted Sobolev-type embedding.
The proofs of the Liouville-type results, Theorems \ref{nonexistence1} and \ref{nonexistence2}, are derived using key estimates proved in Lemmas \ref{Estimativa1} and \ref{Estimativa 2} combined with a contradiction argument.
Theorems \ref{Primeiro Existencia}, \ref{Segundo Existencia}, and \ref{Terceiro Exitencia}, on the existence of nontrivial solutions for \eqref{P}, are proved by using a new weighted Sobolev embedding for the upper half-space and the fibering method. 

\bigskip

\paragraph{Structure of the paper}
In Section 3, we present the proofs of Theorems~\ref{Hardy} and \ref{Sobolev}. 
In Section 4, we establish our nonexistence results by proving Theorems \ref{nonexistence1} and \ref{nonexistence2}. 
Section 5 provides the proofs of Theorems \ref{Primeiro Existencia}, \ref{Segundo Existencia}, and \ref{Terceiro Exitencia}. 
Finally, in Section 6, we comment on potential further developments related to the main subject addressed in this paper.

\section{Proof of Hardy and Sobolev type inequalities}
Before proving our nonexistence and existence results for \eqref{P}, we first establish Theorems ~\ref{Hardy} and \ref{Sobolev}, which play a fundamental role in the sequel. We shall borrow some insights from the paper \cite{MR4189791} to achieve this.

\begin{proof}[Proof of Theorem~\ref{Hardy}]  Let $p>1$ and $\sigma$ be a real number to be chosen later. For any $u\in C_0^\infty(\mathds{R}^N)$, using integration by parts we obtain
\[
(\sigma+1)\int_{\mathds{R}^N_+}(1+x_N)^{\sigma}|u|^p \, \mathrm{d} x +\int_{\mathds{R}^{N-1}}|u|^p\, \mathrm{d} x'=-\int_{\mathds{R}^N_+}(1+x_N)^{\sigma+1}(|u|^p)_{x_N}\, \mathrm{d} x.
\]
On the other hand, 
$$
\begin{aligned}
\left|-\int_{\mathds{R}^N_+}(1+x_N)^{\sigma+1}(|u|^p)_{x_N}\, \mathrm{d} x\right|
			\leq &p\int_{\mathds{R}^N_+}(1+x_N)^{\sigma+1}|u|^{p-1}|\nabla u| \, \mathrm{d} x.
\end{aligned}
 $$
 For $a,b\geq0$, real numbers and $\varepsilon>0$, we can use the Young inequality to get
		$$
		pab\leq \varepsilon(\sigma+1)a^{\frac{p}{p-1}}+\left(\frac{p-1}{\varepsilon(\sigma+1)}\right)^{p-1}b^p.
		$$
Taking into account that 
 $$
 p(1+x_N)^{\sigma+1}|u|^{p-1}|\nabla u|=p(1+x_N)^{(\sigma+1-\frac{\gamma}{p})}|u|^{p-1}(1+x_N)^{\frac{\gamma}{p}}|\nabla u|,
 $$
we  obtain 
		\begin{align*}
			p\int_{\mathds{R}^N_+}(1+x_N)^{(\sigma+1}|u|^{p-1}|\nabla u|\, \mathrm{d} x\leq&\varepsilon(\sigma+1)\int_{\mathds{R}^N_+}(1+x_N)^{(\sigma+1-\gamma/p)\frac{p}{p-1}}|u|^p\, \mathrm{d} x\\
			+&\left(\frac{p-1}{\varepsilon(\sigma+1)}\right)^{p-1}\int_{\mathds{R}^{N_+}}(1+x_N)^\gamma|\nabla u|^p \, \mathrm{d} x.
		\end{align*}
		Combining the above inequalities, we get
		\begin{align*}
			(\sigma+1)\int_{\mathds{R}^N_+}(1+x_N)^{\sigma}|u|^p \, \mathrm{d} x+	\int_{\mathds{R}^{N-1}}|u|^{p}\, \mathrm{d} x'&\leq \varepsilon(\sigma+1)\int_{\mathds{R}^N_+}(1+x_N)^{(\sigma+1-\gamma/p)\frac{p}{p-1}}|u|^p\, \mathrm{d} x\\
			&+\left(\frac{p-1}{\varepsilon(\sigma+1)}\right)^{p-1}\int_{\mathds{R}^{N_+}}(1+x_N)^\gamma|\nabla u|^p \, \mathrm{d} x.
		\end{align*}
		Next, choosing $\sigma$ such that 
		$$
		\left(\sigma+1-\frac{\gamma}{p}\right)\frac{p}{p-1}=\sigma,
		$$
		we have  $\sigma=\gamma-p$ and $\sigma+1=\gamma-p+1>0$. Thus, one has
		$$
		\begin{aligned}	(\sigma+1)(1-\varepsilon)\int_{\mathds{R}^N_+}(1+x_N)^{\sigma}|u|^p\, \mathrm{d} x+&	\int_{\mathds{R}^{N-1}}|u|^{p}\, \mathrm{d} x'\\
			\leq& \left(\frac{p-1}{\varepsilon(\sigma+1)}\right)^{p-1}\int_{\mathds{R}^{N_+}}(1+x_N)^\gamma|\nabla u|^p \, \mathrm{d} x,
   \end{aligned}
		$$
which implies 
  $$
  \begin{aligned}
			(\sigma+1)^p(\varepsilon^{p-1}-\varepsilon^p)\int_{\mathds{R}^N_+}(1+x_N)^{\sigma}|u|^p\, \mathrm{d} x&+(\sigma+1)^{p-1}\varepsilon^{p-1}	\int_{\mathds{R}^{N-1}}|u|^{p}\, \mathrm{d} x'\\
			&\leq (p-1)^{p-1}\int_{\mathds{R}^{N_+}}(1+x_N)^\gamma|\nabla u|^p \, \mathrm{d} x.
\end{aligned}
$$
		Since the function $f(\varepsilon)=\varepsilon^{p-1}-\varepsilon^p$, for $\varepsilon>0$ has it maximum at 
		$$
		\varepsilon_0=\frac{p-1}{p}\quad\mbox{and}\quad f(\varepsilon_0)=(p-1)^{p-1}\frac{1}{p^p},
		$$
		a simple computation shows that  
		\begin{align*}
			\left(\frac{\sigma+1}{p}\right)^p\int_{\mathds{R}^N_+}(1+x_N)^{\sigma}|u|^p\, \mathrm{d} x+&\left(\frac{\sigma+1}{p}\right)^{p-1}	\int_{\mathds{R}^{N-1}}|u|^{p}\, \mathrm{d} x'\leq \int_{\mathds{R}^{N_+}}(1+x_N)^\gamma|\nabla u|^p\, \mathrm{d} x,
		\end{align*}
		and this concludes the proof.
	\end{proof}
	
	\begin{proof}[Proof of Theorem~\ref{Sobolev}] By density, we can assume that $u\in C^\infty_0(\mathds{R}^N)$ and we shall proceed with the proof in several steps.  First, we assume $1<p<N$ and $p-1<\gamma\leq p$. By the classical Gagliardo-Nirenberg-Sobolev inequality, we have
 \begin{equation}\label{GNS}
			\left(\int_{\mathds{R}^N_+}|v|^{p^*}\, \mathrm{d} x\right)^{(N-p)/N}\leq C\int_{\mathds{R}^N_+}|\nabla v|^p\, \mathrm{d} x,~\forall v \in C^{\infty}_0(\mathds{R}^N),
		\end{equation}
  which holds for every $1\leq p<N$. Thus, for $q\in[p,p^*]$, by interpolation inequality there exists $\alpha\in[0,1]$ such that 	
		$$
		\begin{aligned}
	\left(\int_{\mathds{R}^N_+}\frac{|u|^q}{(1+x_N)^{p-\gamma}}\, \mathrm{d} x\right)^{p/q} \leq&\left(\int_{\mathds{R}^N_+} \frac{|u|^p}{(1+x_N)^{p-\gamma}}\, \mathrm{d} x\right)^{\alpha}
\left(\int_{\mathds{R}^N_+} \frac{|u|^{p^*}}{(1+x_N)^{p-\gamma}}\, \mathrm{d} x\right)^{p(1-\alpha)/p^*}\\
			\leq& \left(\int_{\mathds{R}^N_+} \frac{|u|^p}{(1+x_N)^{p-\gamma}}\, \mathrm{d} x\right)^{\alpha}\left(\int_{\mathds{R}^N_+} |u|^{p^*}\, \mathrm{d} x\right)^{p(1-\alpha)/p^*}
   \end{aligned}
   $$
Then by Theorem~\ref{Hardy} and \eqref{GNS}
   $$
   \begin{aligned}
		\left(\int_{\mathds{R}^N_+}\frac{|u|^q}{(1+x_N)^{p-\gamma}}\, \mathrm{d} x\right)^{p/q} 	\leq& C\left(\int_{\mathds{R}^N_+} (1+x_N)^{\gamma}|\nabla u|^p\, \mathrm{d} x\right)^{\alpha}\left(\int_{\mathds{R}^N_+} |\nabla u|^p\, \mathrm{d} x\right)^{1-\alpha}\\
			\leq& C \int_{\mathds{R}^N_+} (1+x_N)^{\gamma}|\nabla u|^p\, \mathrm{d} x.
		\end{aligned}
		$$
		Next, assume $1<p<N$ and $\gamma>p$. Once again, by interpolation, it is sufficient to prove that 
		\begin{equation}\label{p^* inequality}
			\left(\int_{\mathds{R}^N_+}\frac{|v|^{p^*}}{(1+x_N)^{p-\gamma}}\, \mathrm{d} x\right)^{(N-p)/N}\leq C\int_{\mathds{R}^N_+}(1+x_N)^{\gamma}|\nabla v|^p\, \mathrm{d} x,~\forall v \in C^{\infty}_0(\mathds{R}^N).
		\end{equation}
Defining  $v=:u/(1+x_N)^\sigma$ with $u \in C^{\infty}_0(\mathds{R}^N)$ and using a simple computation we see that 
		\begin{align*}
			\nabla v=&\frac{1}{(1+x_N)^\sigma} \left(\nabla u-(0',\frac{\sigma u}{(1+x_N)})\right)
		\end{align*}
		and consequently, there exists a constant $C=C(p,\sigma)>0$ such that 
		$$
			|\nabla v|^p \leq C\left(\frac{|\nabla u|^p}{(1+x_N)^{p\sigma}}+\frac{|u|^p}{(1+x_N)^{(\sigma+1)p}}\right).
		$$
		This, together with \eqref{GNS}, implies that 
		\begin{equation}\label{GNS v}
			\left(\int_{\mathds{R}^N_+}\frac{|u|^{p^*}}{(1+x_N)^{\sigma p^*}}\, \mathrm{d} x\right)^{(N-p)/N}\leq C\int_{\mathds{R}^N_+}\left(\frac{|\nabla u|^p}{(1+x_N)^{p\sigma}}+\frac{|u|^p}{(1+x_N)^{(\sigma+1)p}}\right)\, \mathrm{d} x.
		\end{equation}
		Choosing $\sigma<0$ such that $\sigma p^*=p-\gamma$ we deduce that  
		\[
		-\sigma p=\frac{(\gamma-p)(N-p)}{N}<\gamma-p<\gamma
		\]
		and hence
		$
		(\sigma+1)p>p-\gamma.
		$
		Thus, from  \eqref{GNS v} and the Hardy inequality \eqref{hardy with xn} we  conclude that  \eqref{p^* inequality} holds.\\
	
 Now assume that $p=N$. For $u \in C^{\infty}_0(\mathds{R}^N)$, applying inequality \eqref{GNS} with $p=1$ and $v=\frac{|u|^N}{(1+x_N)^\sigma}$  we get
	\begin{align*}
\left(\int_{\mathds{R}^N_+}\frac{|u|^{N^2/(N-1)}}{(1+x_N)^{\frac{\sigma N}{N-1}}}\, \mathrm{d} x\right)^{(N-1)/N}\leq &C|\sigma|\int_{\mathds{R}^{N}_+}\frac{|u|^N}{(1+x_N)^{\sigma+1}}\, \mathrm{d} x+CN\int_{\mathds{R}^{N}_+}\frac{|u|^{N-1}|\nabla u|}{(1+x_N)^{\sigma}}\, \mathrm{d} x.
\end{align*}
Choosing $\sigma+1=N-\gamma$ and 
		using Young's inequality, we obtain
		\begin{align*}
			\int_{\mathds{R}^{N}_+}\frac{|u|^{N-1}|\nabla u|}{(1+x_N)^{\sigma}}\, \mathrm{d} x=&\int_{\mathds{R}^{N}_+}\frac{|u|^{N-1}(1+x_N)^{\gamma/N}|\nabla u|}{(1+x_N)^{\sigma+\gamma/N}}\, \mathrm{d} x\\
			\leq& C\int_{\mathds{R}^{N}_+}\frac{|u|^N}{(1+x_N)^{N-\gamma}}+(1+x_N)^{\gamma}|\nabla u|^N\, \mathrm{d} x,
		\end{align*}
		where we used that 
		\begin{equation}\label{sigma equation}
			\left(\sigma+\frac{\gamma}{N}\right)\frac{N}{N-1}= N-\gamma.
		\end{equation}
Since $
		\frac{\sigma N}{N-1}\leq N-\gamma,
		$
	using Theorem~\ref{Hardy} we get
		\begin{equation}\label{key 1}
  \left(\int_{\mathds{R}^N_+}\frac{|u|^{N^2/(N-1)}}{(1+x_N)^{N-\gamma}}\, \mathrm{d} x\right)^{(N-1)/N}\leq C \int_{\mathds{R}^{N}_+}(1+x_N)^\gamma|\nabla u|^N\, \mathrm{d} x.
		\end{equation}
		Thus, interpolation implies
		\begin{align*}\left(\int_{\mathds{R}^N_+}\frac{|u|^q}{(1+x_N)^{N-\gamma}}\, \mathrm{d} x\right)
			^{N/q}\leq& \left(\int_{\mathds{R}^{N}_+}
			\frac{|u|^N}{(1+x_N)^{N-\gamma}}\, \mathrm{d} x\right)^{\theta}\left(\int_{\mathds{R}^{N}_+}\frac{|u|^{N^2/(N-1)}}{(1+x_N)^{N-\gamma}}\, \mathrm{d} x\right)^{(1-\theta)(N-1)/{N}}\\
			\leq& C\int_{\mathds{R}^{N}_+}(1+x_N)^\gamma|\nabla u|^N\, \mathrm{d} x,
		\end{align*}
		for any $q \in [N,N^2/(N-1)]$. 	Since $N<N+1<N^2/(N-1)$, in particular we get
		\begin{equation}\label{q=N+1 2}
			\left(\int_{\mathds{R}^{N}_+}\frac{|u|^{N+1}}{(1+x_N)^{N-\gamma}}\, \mathrm{d} x\right)^{N/(N+1)}\leq C\int_{\mathds{R}^{N}_+}(1+x_N)^{\gamma}|\nabla u|^N\, \mathrm{d} x.
		\end{equation}
		Once again, applying ~\eqref{GNS} with $p=1$ and $v=\frac{|u|^{N+1}}{(1+x_N)^{\sigma}}$ and using Young's inequality we have
		\begin{align*}
			\left(\int_{\mathds{R}^{N}_+}\frac{|u|^{\frac{N(N+1)}{N-1}}}{(1+x_N)^{\frac{\sigma N}{N-1}}}\, \mathrm{d} x\right)^{(N-1)/N} \leq&C\left(\int_{\mathds{R}^{N}_+}\frac{|u|^{N+1}}{(1+x_N)^{\sigma+1}}\, \mathrm{d} x +\int_{\mathds{R}^{N}_+}\frac{|u|^N|\nabla u|}{(1+x_N)^{\sigma}}\, \mathrm{d} x\right).
		\end{align*}
		Using H\"{o}lder's inequality and~\eqref{sigma equation}, we obtain
		\begin{align*}
			\int_{\mathds{R}^{N}_+}\frac{|u|^{N}|\nabla u|}{(1+x_N)^{\sigma}}\, \mathrm{d} x=&\int_{\mathds{R}^{N}_+}\frac{|u|^{N}(1+x_N)^{\gamma/N}|\nabla u|}{(1+x_N)^{\sigma+\gamma/N}}\, \mathrm{d} x\\
			\leq&\left(\int_{\mathds{R}^{N}_+}\frac{|u|^{\frac{N^2}{N-1}}}{(1+x_N)^{N-\gamma}}\, \mathrm{d} x\right)^{(N-1)/N}\left(\int_{\mathds{R}^{N}_+}(1+x_N)^{\gamma}|\nabla u|^N\, \mathrm{d} x\right)^{1/N}.
		\end{align*}
		Thus, from \eqref{key 1} and \eqref{q=N+1 2} one has
		\begin{equation*}
			\left(\int_{\mathds{R}^{N}_+}\frac{|u|^{\frac{N(N+1)}{N-1}}}{(1+x_N)^{N-\gamma}}\, \mathrm{d} x\right)^\frac{N-1}{N+1}\leq C \int_{\mathds{R}^{N}_+}(1+x_N)^{\gamma}|\nabla u|^N\, \mathrm{d} x
		\end{equation*}
		and interpolation implies
		\begin{equation*}
			\left(\int_{\mathds{R}^{N}_+}\frac{|u|^q}{(1+x_N)^{N-\gamma}}\, \mathrm{d} x\right)^{N/q}\leq C \int_{\mathds{R}^{N}_+}(1+x_N)^{\gamma}|\nabla u|^N\, \mathrm{d} x,
		\end{equation*}
		for any $q \in[N,\frac{N(N+1)}{N-1}]$. Reiterating this argument with $k=N+2,N+3,\ldots,$ we get
		\[
		\left(\int_{\mathds{R}^{N}_+}\frac{|u|^{Nk/(N-1)}}{(1+x_N)^{N-\gamma}}\, \mathrm{d} x\right)^{(N-1)/k} \leq C\int_{\mathds{R}^{N}_+}(1+x_N)^\gamma|\nabla u|^N\, \mathrm{d} x.
		\]
		Now, given $q\in (N,\infty)$, we can choose $k\geq N$ such that $q \in(N,Nk/(N-1))$ and by interpolation, we can conclude the proof.
	\end{proof}
	\section{Proof of Liouville-type Results}
	%%%%%%%%%%%%%%%%%%%%%%%%%%%%%%%%%%%%%%%
In this section, we shall focus on proving our Liouville-type results. The following estimate is instrumental in our analysis.

\begin{lemma}[key estimate]\label{Estimativa1} Assume condition $(H_0)$ and $1<s<q <p$. 
If $b\in L^1(\mathds{R}^N_+)\cap \mathcal{K}_b$, then  $B$ is well defined in $E$. 
In addition, if $a/b\in L^{\infty}(\mathds{R}^N_+)$, then 
		\begin{equation}\label{Chave1}
			A(v)^{p-s}\leq \left(\left\|\frac{a}{b}\right\|_{\infty}^{(p-s)}	\left[\frac{b_0C_{p,\gamma}^{-p}}{\rho_0}\right]^{q-s}\right)B(v)^{p-q}\|v\|^{p(q-s)},
   \quad \forall v\in  E.
		\end{equation}
		In particular, $A$ is well-defined. Furthermore, if \eqref{condition for nonexistence} holds, then  
		\begin{equation}\label{Estima2}
			A(v)^{p-s}<\eta(s,q,p)B(v)^{p-q}\|v\|^{p(q-s)},    \quad \forall v\in E\backslash\{0\}.
		\end{equation}
\end{lemma}
 \begin{proof}
 Since $b\in\mathcal{K}_b$, we have $b(x)\leq b_0/(1+x_N)^{p-\gamma}$. 
 Thus, by H\"{o}lder's inequality, 
		\begin{align*}
			B(v)=\int_{\mathds{R}^{N}_+}b^{\frac{p-s}{p}}b^{s/p}|v|^s\, \mathrm{d} x
			\leq& \|b\|_{L^1(\mathds{R}^N_+)}^{(p-s)/s}\left(\int_{\mathds{R}^{N}_+}b|v|^p\, \mathrm{d} x\right)^{s/p}\\
			\leq& b_0^{s/p}\|b\|_{L^1(\mathds{R}^N_+)}^{(p-s)/s}\left(\int_{\mathds{R}^{N}_+}\frac{|v|^p}{(1+x_N)^{p-\gamma}}\, \mathrm{d} x\right)^{s/p},
		\end{align*}
which in finite by Theorem~\ref{Hardy} and assumption $(H_0)$.  Since  $a$ is nonnegative, we get 
$a(1+x_N)^{p-\gamma}\leq b_0a/b$, which implies 
 \begin{equation}\label{boa1}
 \|a(1+x_N)^{p-\gamma}\|_\infty\leq b_0\left\|\frac{a}{b}\right\|_\infty .
 \end{equation}
 If $s<q<p$ we can write $q=(1-\alpha)s+\alpha p$ with $\alpha=(q-s)/(p-s)\in (0,1)$ . 
 Thus, by H\"{o}lder's inequality,
		\[
A(v)=\int_{\mathds{R}^{N}_+}a|v|^q\, \mathrm{d} x=\int_{\mathds{R}^{N}_+} (a|v|^s)^{1-\alpha}(a|v|^p)^\alpha \, \mathrm{d} x\leq \left(\int_{\mathds{R}^{N}_+}a|v|^s\, \mathrm{d} x\right)^{1-\alpha}\left(\int_{\mathds{R}^{N}_+} a|v|^p\, \mathrm{d} x\right)^{\alpha}.
\]
Using that $1-\alpha=(p-q)/(p-s)$ we obtain
\begin{equation}\label{A ineq}
			A(v)^{p-s}\leq \left(\int_{\mathds{R}^{N}_+}a|v|^s\, \mathrm{d} x\right)^{p-q}\left(\int_{\mathds{R}^{N}_+} a|v|^p\, \mathrm{d} x\right)^{q-s}.
   \end{equation}
Now, observe that 
   $$
   \int_{\mathds{R}^{N}_+} a|v|^s\, \mathrm{d} x=\int_{\mathds{R}^{N}_+} \frac{a}{b}\left(b|v|^s\right)\, \mathrm{d} x\leq\left\|\frac{a}{b}\right\|_\infty B(v).
   $$
   Thus, \eqref{boa1} and Theorem~\ref{Hardy} gives
   $$
   \int_{\mathds{R}^{N}_+} a|v|^p\, \mathrm{d} x\leq\|a(1+x_N)^{p-\gamma}\|_\infty\int_{\mathds{R}^{N}_+}\frac{|v|^p}{(1+x_N)^{p-\gamma}}\, \mathrm{d} x\leq b_0\left\|\frac{a}{b}\right\|_\infty \frac{C^{-p}_{p,\gamma}}{\rho_0}\|v\|^p.
   $$
Therefore, plugging the last two inequalities into \eqref{A ineq}, we estimate \eqref{Chave1}.
\end{proof}

Throughout the paper, we will consider the following auxiliary function:
\begin{equation}\label{G definition}
			G(r,v)=A(v)r^{q-p}-B(v)r^{s-p},\quad r>0\quad\mbox{and}\quad v\in E.
		\end{equation}	
\begin{lemma}\label{Max G}  Assume the assumptions in Theorem \ref{nonexistence1}. 
For each fixed $v \in E\backslash\{0\}$ the function $G(.,v)$  has a unique critical point which is a maximum and is given by	
		\begin{equation}\label{critical point r(v)}
			\overline{r}(v)=\left(\frac{B(v)(p-s)}{A(v)(p-q)}\right)^{1/(q-s)}.
		\end{equation}
		Moreover, 
		\begin{equation}\label{G(r,v)}
			G(\overline{r}(v),v)=	\max_{r>0}G(r,v)=\left(\frac{A(v)^{p-s}}{\eta(s,q,p)B(v)^{p-q}}\right)^{1/(q-s)}>0,
   \end{equation}
   where $\eta(s,q,p)$ was defined in \eqref{GeneralConstant}.
	\end{lemma}
\begin{proof} Note that for each $v \in E\backslash\{0\},$ we have
 $$
\frac{\partial G}{\partial r}(r,v)=(q-p)A(v)r^{q-p-1}-(s-p)B(v)r^{s-p-1}.
$$
 Thus, 
 $$
\frac{\partial G}{\partial r}(r,v)=0 \iff r=\overline{r}(v)=\left(\frac{B(v)(p-s)}{A(v)(p-q)}\right)^{1/(q-s)}.
$$
Moreover, we can see that $\lim_{r \rightarrow +\infty} G(r,v)=0$ and $\lim_{r\rightarrow 0^+}G(r,v)=-\infty$. Since, 
\begin{align*}
			G(\overline{r}(v),v)
   &=A(v)\overline{r}(v)^{q-p}-B(v)\overline{r}(v)^{s-p}\\
						&= \left(\frac{q-s}{p-q}\right)B(v)\overline{r}(v)^{s-p}\\
			&=\left(\frac{q-s}{p-q}\right)B(v) \left(\frac{B(v)(p-s)}{A(v)(p-q)}\right)^{(s-p)/(q-s)}\\
			&=\left(\frac{A(v)^{p-s}}{\eta(s,q,p) B(v)^{p-q}}\right)^{1/(q-s)}>0
		\end{align*}
we can conclude that $G(.,v)$ has a unique global maximum at $r=\overline{r}(v)>0$. 	\end{proof}
	
	Now we are ready to present the proof of Theorem \ref{nonexistence1}.
	
	\begin{proof}[Proof of Theorem \ref{nonexistence1}]
 Assume by contradiction that  \eqref{P} has a nontrivial weak solution $u_0\in E$. Then, from the definition \eqref{weak solution}, Lemma~\ref{Estimativa1} and a density argument imply
		$$
			\|u_0\|^p=A(u_0)-B(u_0)=G(1,u_0).
		$$
On the other hand, by estimate \eqref{Estima2}, we have 
		\begin{equation*}
			A(u_0)^{p-s}<\eta(s,q,p)B(u_0)^{p-q}\|u_0\|^{p(q-s)},
		\end{equation*}
which combined with \eqref{G(r,v)} gives
	$
			G(\overline{r}(u_0),u_0)<\|u_0\|^p.
	$
Thus, we get 
	$$
 G(\overline{r}(u_0),u_0)<\|u_0\|^p=G(1,u_0),
 $$
 which contradicts the fact that $\overline{r}(u_0)$ is the maximum of $G(.,u_0)$  and this concludes  the proof of Theorem~\ref{nonexistence1}.
	\end{proof}
		
		Next, we shall focus on the proof of our second Liouville-type result.
 \begin{lemma}\label{Estimativa 2} Assume condition $(H_0)$, $p<q<s\leq p^*$ for $p<N$ and $p<q<s< \infty$ for $p=N$. If $b\in\mathcal{K}_b$, then $B$ is well defined. In addition, if $a/b\in L^{\infty}(\mathds{R}^N_+)$, then 
		\begin{equation}\label{est 2}
			A(v)^{s-p}\leq \left(\left\|\frac{a}{b}\right\|_{\infty}^{s-p}	\left[\frac{b_0C_{p,\gamma}^{-p}}{\rho_0}\right]^{s-q}\right)B(v)^{q-p}\|v\|^{p(s-q)}, \quad \forall v \in E.
		\end{equation}
		In particular, $A$ is well-defined. Furthermore, if \eqref{condition for nonexistence 2} holds then  
		\begin{equation}\label{Estima3}
			A(v)^{s-p}<\eta(p,q,s) B(v)^{q-p}\|v\|^{p(s-q)},  \quad \forall v \in E\backslash\{0\}.
		\end{equation}
	\end{lemma}
\begin{proof} If $b\in \mathcal{K}_b$ and $v\in E$ we see that 
 $$
B(v)=\int_{\mathds{R}^{N}_+} b|v|^s\, \mathrm{d} x\leq b_0\int_{\mathds{R}^{N}_+}\frac{|v|^s}{(1+x_N)^{p-\gamma}}\, \mathrm{d} x,
$$
which is finite thanks to assumption $(H_0)$ and Theorem~\ref{Sobolev}. Using again that $b\in\mathcal{K}_b$ and $a$ is nonnegative we have
 $a(1+x_N)^{p-\gamma}\leq b_0a/b$, which implies 
 $$
 \|a(1+x_N)^{p-\gamma}\|_\infty\leq b_0\left\|\frac{a}{b}\right\|_\infty.
 $$
 Since $p<q<s$ we can write $q=(1-\alpha)p+\alpha s$ with $\alpha=(q-p)/(s-p)\in (0,1)$ . Thus, by H\"{o}lder's inequality we get 
		\[
A(v)=\int_{\mathds{R}^{N}_+}a|v|^q\, \mathrm{d} x=\int_{\mathds{R}^{N}_+} (a|v|^p)^{1-\alpha}(a|v|^s)^\alpha \, \mathrm{d} x\leq \left(\int_{\mathds{R}^{N}_+}a|v|^p\, \mathrm{d} x\right)^{1-\alpha}\left(\int_{\mathds{R}^{N}_+} a|v|^s\, \mathrm{d} x\right)^{\alpha}.
		\]
Taking into account that $1-\alpha=(s-q)/(s-p)$ we obtain
\begin{equation}\label{A ineqB}
			A(v)^{s-p}\leq \left(\int_{\mathds{R}^{N}_+}a|v|^p\, \mathrm{d} x\right)^{s-q}\left(\int_{\mathds{R}^{N}_+} a|v|^s\, \mathrm{d} x\right)^{q-p}.
   \end{equation}
   Now, thanks to Theorem~\ref{Hardy} and \eqref{boa1} we get 
   $$
   \int_{\mathds{R}^{N}_+} a|v|^p\, \mathrm{d} x\leq\|a(1+x_N)^{p-\gamma}\|_\infty\int_{\mathds{R}^{N}_+}\frac{|v|^p}{(1+x_N)^{p-\gamma}}\, \mathrm{d} x\leq b_0\left\|\frac{a}{b}\right\|_\infty \frac{C^{-p}_{p,\gamma}}{\rho_0}\|v\|^p
   $$
   and notice that 
   $$
   \int_{\mathds{R}^{N}_+} a|v|^s\, \mathrm{d} x=\int_{\mathds{R}^{N}_+} \frac{a}{b}\left(b|v|^s\right)\, \mathrm{d} x\leq\left\|\frac{a}{b}\right\|_\infty B(v).
   $$
Therefore, plugging the last two inequalities into \eqref{A ineqB}, we obtain estimate \eqref{est 2}.
\end{proof}

Arguing along the same lines as in the proof of Lemma~\ref{Max G}, we can obtain the following result:

\begin{lemma}  Assume condition $(H_0)$, $p<q<s\leq p^*$ for $p<N$ and $p<q<s< \infty$ for $p=N$. If $b\in\mathcal{K}_b$ and $a/b \in L^\infty(\mathds{R}^N_+)$, then for each $v \in E\backslash\{0\}$  the function $G(.,v)$ defined by \eqref{G definition} has a unique critical point at 	
		\begin{equation}\label{critical point r(v) 2}
			\overline{r}(v)=\left(\frac{A(v)(q-p)}{B(v)(s-p)}\right)^{1/(s-q)}.
		\end{equation}
		Moreover, 
		\begin{equation}\label{G(r,v) 2}
			G(\overline{r}(v),v)=	\max_{r>0}G(r,v)=\left(\frac{A(v)^{s-p}}{\eta(p,q,s) B(v)^{q-p}}\right)^{1/(s-q)}>0.
		\end{equation}
		\end{lemma}

\begin{proof}[Proof of Theorem \ref{nonexistence2}:] 
Arguing by contradiction, suppose that \eqref{P} has a nontrivial weak solution $u_0\in E$. From the definition \eqref{weak solution}, Lemma \ref{Estimativa 2} and a density argument we have
\begin{equation}\label{Existence2}
			\|u_0\|^p=A(u_0)-B(u_0)=G(1,u_0).
		\end{equation}
On the other hand, by estimate \eqref{Estima3}, we have 
		\begin{equation*}
			A(u_0)^{s-p}<\eta(p,q,s)B(u_0)^{q-p}\|u_0\|^{p(s-q)}.
		\end{equation*}
		This, together with \eqref{G(r,v) 2} implies that 
	$
G(\overline{r}(u_0),u_0)<\|u_0\|^p.
	$
Therefore, we obtain 
$$
G(\overline{r}(u_0),u_0)<\|u_0\|^p=G(1,u_0),
$$
contradicting the fact that $\overline{r}(u_0)$ is the maximum of $G(.,u_0)$ and this concludes the proof.
\end{proof}

	\section{Existence Results}
	This section is devoted to proving Theorems \ref{Primeiro Existencia}, \ref{Segundo Existencia} and \ref{Terceiro Exitencia}. To this purpose, we shall first prove a compactness result.

\begin{lemma}\label{compactness} Assume condition $(H_0)$ and $1<p\leq N$. 
\begin{enumerate}
    \item If $k\in L^1(\mathds{R}^N_+)\cap \mathcal{K}_b$, then the embedding 
    \begin{equation}\label{Mae}
		E\hookrightarrow L^q\left(\mathds{R}^N_+,k(x)\right) 
	\end{equation}
 is compact for all $1<q<p\leq N$.
 \item If $k\in \mathcal{K}_0$ and  $p<N$, then the embedding \eqref{Mae}	is continuous for  $q \in [p,p^*]$ and compact for $q \in [p,p^*)$. If $p=N$, the embedding is  compact for all $q\in[p,\infty)$.
\end{enumerate}
\end{lemma} 

\begin{proof} 
If $k \in L^1(\mathds{R}^N_+)$, by H\"{o}lder's inequality,
\begin{align*}
\int_{\mathds{R}^N_+}k|u|^q\, \mathrm{d} x=\int_{\mathds{R}^N_+}k^{\frac{p-q}{p}}k^{q/p}|u|^q\, \mathrm{d} x
    \leq\|k\|_1^{\frac{p-q}{p}}\left(\int_{\mathds{R}^N_+} k|u|^p\, \mathrm{d} x\right)^{q/p}.
\end{align*}
Also, since $k\in\mathcal{K}_b$ we have $k(x)\leq k_0(1+x_N)^{\gamma-p}$ and  by assumption $(H_0)$ and  Theorem \ref{Hardy} we obtain
\begin{align*}
\int_{\mathds{R}^N_+}k|u|^q \, \mathrm{d} x\leq C\|k\|_1^{\frac{p-q}{p}}\left(\int_{\mathds{R}^N_+}\frac{|u|^p}{(1+x_N)^{p-\gamma}} \, \mathrm{d} x \right)^{q/p}
\leq C\|k\|_1^{\frac{p-q}{p}}\|u\|^q.
\end{align*}
Now, if $(u_n) \subset E$ is a bounded sequence, up to a subsequence, we can assume that $u_n \rightharpoonup 0$ in $E$. Given $\varepsilon>0$ there exists $R=R(\varepsilon)>0$ such that $\|k\|_{L^1(B^c_R(0)\cap\mathds{R}^N_+)}\leq\varepsilon$ and hence 
\begin{align*}
\int_{B_R^c\cap \mathds{R}^N_+}k|u_n|^q\, \mathrm{d} x\leq  C\varepsilon^{\frac{p-q}{p}}\|u_n\|^q\leq C_1\varepsilon^{\frac{p-q}{p}}.
\end{align*}
To complete the proof for the first case, it is enough to use the classical Sobolev compact embedding to obtain the compact embedding $E\hookrightarrow W^{1,p}\hookrightarrow L^q(B_R^+)$.

Assuming  $k\in \mathcal{K}_0$, we have that $k(x)\leq C_0(1+x_N)^{\gamma-p}$ for some constant $C_0>0$, which implies that the embedding is continuous by Theorem \ref{Sobolev} and the assumption $(H_0)$ if $1<p<N$ and $q\in [p,p^*]$ or $q\in [p,\infty)$ if $p=N$. 
For $R>0$, we can write 
\begin{align*}
		\int_{\mathds{R}^N_+}k(x)|u|^q\, \mathrm{d} x=\int_{B_R^+}k(x)|u|^q\, \mathrm{d} x+\int_{(B_R^+)^c\cap \mathds{R}^N_+}k(x)|u|^q\, \mathrm{d} x.
	\end{align*}
If $(u_n) \subset E$ is a bounded sequence, up to a subsequence,  $u_n \rightharpoonup 0$ in $E$. Since the embedding $E\hookrightarrow W^{1,p}(B_R^+)\hookrightarrow L^q(B_R^+)$ is compact for all  $q\in [p,p^*]$ if  $1<p<N$ or $q\in [p,\infty)$ if $p=N$, it holds 
	\begin{equation}\label{abacate}
		\int_{B_R^+}k(x)|u_n|^q\, \mathrm{d} x \leq C\int_{B_R^+}|u_n|^q\, \mathrm{d} x \rightarrow 0.
	\end{equation}
Given $\varepsilon>0$, since $k\in\mathcal{K}_0 $ we can choose $R=R(\varepsilon)>0$ large enough such that $k(x)(1+x_N)^{p-\gamma}<\varepsilon$ for any $x\in B_R^c\cap\mathds{R}^N_+$, which implies
	\begin{align}\label{banana}
		\int_{(B_R^+)^c\cap \mathds{R}^N_+}k(x)|u_n|^q\, \mathrm{d} x <&  \varepsilon   \int_{(B_R^+)^c\cap \mathds{R}^N_+}\frac{|u_n|^q}{(1+x_N)^{p-\gamma}}\, \mathrm{d} x\leq C\varepsilon\|u_n\|^q.
	\end{align}
The proof of the second case follows from \eqref{abacate}-\eqref{banana}.\end{proof}

To establish our existence results, let us consider  the functional $I: E\rightarrow \mathds{R}$ associated with  \eqref{P}, defined as follows:
$$
I(u)=\frac{1}{p}\|u\|^p-\frac{1}{q}A(u)+\frac{1}{s}B(u),
$$
where $A$ and $B$ are defined in \eqref{FuntionalAB}. 
Straightforward computation shows that $I\in C^1(E,\mathds{R})$ and critical points of $I$ are weak solutions of \eqref{P}, see \cite{zbMATH05114882}. 
To prove that $I$ has a critical point, we shall use the fibering method \cite{zbMATH01060125,zbMATH05635193}. To this end, we proceed with some basic results.

\begin{lemma}\label{existence of a function r C1}Let $1<p\leq N$ and $q<\min\{s,p\}$ or $q>\max\{s,p\}$. Then, for each $v \in E\backslash\{0\}$ there exits a unique real number $r(v)>0$ such that the pair $(r(v),v)$ satisfies the equation 
  \begin{equation}\label{critical condition}
		\|v\|^p=r(v)^{q-p}A(v)-r(v)^{s-p}B(v)=G(r(v),v).
	\end{equation}
 Furthermore, the map $r:E\backslash\{0\}\rightarrow \mathds{R}$ belongs to $C^1(E\backslash\{0\},\mathds{R})$ and $\mu r(\mu v)=r(v)$ for all $\mu >0$ and $v \in E\backslash\{0\}$.
\end{lemma}

\begin{proof} {\it Existence:}
Consider the function $f:(0,\infty)\times E\rightarrow \mathds{R}$ defined by 
	\begin{equation*}
		f(r,v)=\|v\|^pr^{p-q}+B(v)r^{s-q}-A(v),
	\end{equation*} 	
and note that $f(r,v)=0$ if and only if \eqref{critical condition} holds. \\
If $v \in E\backslash\{0\}$ and $q>\max\{s,p\}$ we have
 \[
	\lim_{r\rightarrow 0^+}f(r,v)=\infty\quad \text{and}\quad \lim_{r \rightarrow +\infty}f(r,v)=-A(v)<0 .
\] 
In the case  $q<\min\{s,p\},$ it holds
 \[
\lim_{r \rightarrow 0^+}f(r,v)=-A(v)<0\quad \text{and}\quad 	\lim_{r\rightarrow \infty}f(r,v)=\infty.
	\]
Thus, in any case, by the Intermediate Value Theorem, $r(v)>0$ exists such that $f(r(v),v)=0$.

\noindent {\it Uniqueness:} 
Fixed $v \in E\backslash\{0\}$,  suppose that there are $r_1,r_2>0$ satisfying  \eqref{critical condition}. Consequently,  
\begin{equation*}
		\|v\|^pr_1^{p-q}+B(v)r_1^{s-q}=A(v)=\|v\|^pr_2^{p-q}+B(v)r_2^{s-q},
	\end{equation*}
which is equivalent to
	\begin{equation*}
		\|v\|^p(r_1^{p-q}-r_2^{p-q})+B(v)(r_1^{s-q}-r_2^{s-q})=0.
	\end{equation*}
Therefore, $r_1=r_2$ and so the map  $r:E\backslash\{0\}\rightarrow \mathds{R}$ satisfying  \eqref{critical condition} is well defined. 

\noindent{\it Regularity:} To prove that $r$ belongs to class $C^1$, we observe that  
	\begin{equation*}
		\frac{\partial f}{\partial r}(r,v)=(p-q)r^{p-q-1}\|v\|^p+(s-q)B(v)r^{s-q-1}\neq 0, \quad \text{in} \quad (0,\infty)\times E\backslash\{0\}.
	\end{equation*}
Using the implicit function theorem, 
 we obtain open sets $J \subset \mathds{R}$  and $V\subset E\backslash\{0\}$ containing $r(v)$ and $v$ respectively, and a $C^1$-function $\tau:V\rightarrow J$ satisfying 
	\[
	\tau(v)=r(v)\quad \text{and}\quad f(\tau(w),w)=0, \quad\forall w \in V.
	\]
By the uniqueness $r\equiv\tau$ in $V$ and therefore $r$ is a $C^1$-function in $V$. 
	
Finally, given $\mu>0$ and $v \in E\backslash\{0\}$ we have that $f(r(\mu v),\mu v)=0$, that is,
	\begin{equation}\label{key 3}
		A(v)=\mu^{p-q}r(\mu v)^{p-q}\|v\|^p+\mu^{s-q}r(\mu v)^{q-s}B(v).
	\end{equation}
Since $f(r(v),v)=0$, we have
$$
		A(v)=r(v)^{p-q}\|v\|^p+r(v)^{q-s}B(v),
  $$
which combined with \eqref{key 3} implies
	\[
	0=(\mu^{p-q}r(\mu v)^{p-q}-r(v)^{p-q})\|v\|^p+(\mu^{s-q}r(\mu v)^{s-q}-r(v)^{s-q})B(v).
	\]
Thus, $(\mu r(\mu v))^{p-q}=r(v)^{p-q}$ and this concludes the proof.  	
\end{proof}

\begin{remark}\label{key remark}
    Suppose that there exists an open $\Omega\subset E\backslash\{0\}$ and $r \in C^1(\Omega, \mathds{R})$ such that $(r(v),v)$ satisfies \eqref{critical condition} for each $v \in \Omega$ with $r(v)\neq 0$ in $\Omega$, that is,
     \begin{equation}\label{critical condition-P}
		\|v\|^p=r(v)^{q-p}A(v)-r(v)^{s-p}B(v).
	\end{equation}
    Then, we have	
	\begin{align*}
		I(r(v)v)=&\frac{r(v)^p}{p}\|v\|^p-\frac{r(v)^q}{q}A(v)+\frac{r(v)^s}{s}B(v)\\
			=&\left(\frac{1}{p}-\frac{1}{q}\right)A(v)r(v)^q+\left(\frac{1}{s}-\frac{1}{p}\right)B(v)r(v)^s.
	\end{align*}
 In particular,  if $r>0$ and $rv$ is a critical point of $I$, it holds
$$
		\langle I'(rv),v\rangle=0,
$$
	which is equivalent to \eqref{critical condition-P}.
\end{remark}
The above remark motivates us to consider the \textit{reduced functional} defined by
\begin{equation}\label{reduced functional}
	\mathcal{I}(v):=I(r(v)v)=\left(\frac{1}{p}-\frac{1}{q}\right)A(v)r(v)^q+\left(\frac{1}{s}-\frac{1}{p}\right)B(v)r(v)^s.
\end{equation}

Next, we shall need the following result to characterize the fibering method.

\begin{lemma}\label{fibering method }Let $H\in C^1(E\backslash\{0\},\mathds{R})$ such that $\langle H'(v),v\rangle \neq 0$ if $H(v)=1$. If $v_c \in \Omega$ is a critical point of $\mathcal{I}$ under the constraint $H(v)=1$, then $u=r(v_c)v_c$ is a critical point of $I$.
\end{lemma}

\begin{proof} Let $r\in C^1(\Omega, \mathds{R})$ as in  Remark \ref{key remark}, that is, for each $v \in \Omega\subset E\backslash \{0\}|$ the pair $(r(v),v)$ satisfies \eqref{critical condition}, more specifically
\begin{equation*}
    \|v\|^p=r(v)^{q-p}A(v)-r(v)^{s-p}B(v).
\end{equation*}
Then we can define $\mathcal{I}:\Omega \rightarrow \mathds{R}$ as in \eqref{reduced functional} and 
\begin{equation}\label{derivative app in r(v)}
    \langle I'(r(v)v),v\rangle=0, \forall v \in \Omega.
\end{equation}
In fact,
\begin{align*}
    \langle I'(r(v)v),v\rangle=& r(v)^{p-1} \|v\|^p-r(v)^{q-1}A(v)+r(v)^{s-1}B(v)\\
    =&r(v)^{p-1}[\|v\|^p-r(v)^{q-p}A(v)+r(v)^{s-p}B(v)]=0
\end{align*}
If $v_c$ is a critical point of $\mathcal{I}$ under the constraint $H(v)=1$, by the Lagrange Multiplier Theorem, there exists $\lambda \in \mathds{R}$ such that 
	\begin{equation}\label{conditional cp}
		\mathcal{I}'(v_c)=\lambda H'(v_c).
	\end{equation}
On the other hand, by the definition of $\mathcal{I}$ and \eqref{derivative app in r(v)} we have 
	\begin{equation}\label{derivative of reduced functional}
		\langle \mathcal{I}'(v),w\rangle=r(v)\langle I'(r(v)v),w \rangle+\langle r'(v),w\rangle \langle I'(r(v)v),v\rangle= r(v)\langle I'(r(v)v),w \rangle
	\end{equation}
	for all $w \in E$.
Then by \eqref{derivative app in r(v)} and \eqref{conditional cp} 
\begin{align*}
  0=r(v_c)  \langle {I}'(r(v_c)v_c),v_c\rangle=\langle \mathcal{I}'(v_c),v_c\rangle=\lambda \langle H'(v_c),v_c\rangle.
\end{align*}
Since $\langle H'(v_c),v_c\rangle\neq0$ we have that $\lambda=0$ and hence, by \eqref{conditional cp} and \eqref{derivative of reduced functional},
 	\[
	0=\mathcal{I}'(v_c)= r(v_c)I'(r(v_c)v_c).
 \]
 Therefore, $r(v_c)v_c$ is a critical point of $I$.
\end{proof}

Now, we are ready to proceed with the proof of Theorems \ref{Primeiro Existencia} and \ref{Segundo Existencia}.  

\begin{proof}[Proof of Theorem \ref{Primeiro Existencia}:] 
	For each fixed $v \in E\backslash\{0\}$, by Lemma \ref{existence of a function r C1}, there exists $r(v)>0$ such that the pair $(r(v),v)$ satisfies \eqref{critical condition} and hence 
 \begin{equation}%\label{c condition 2} % Precisa citar essa equação mais a frente.
     \|v\|^pr(v)^{p-q}+B(v)r(v)^{s-q}=A(v).
 \end{equation}
As a consequence, we can consider the reduced functional $\mathcal{I}$ as
	\begin{align*}
		\mathcal{I}(v)
		= \left(\frac{1}{s}-\frac{1}{q}\right)B(v)r(v)^s+\left(\frac{1}{p}-\frac{1}{q}\right)\|v\|^pr(v)^p>0.
	\end{align*}
If $S^1$ denotes the unity sphere in $E$, we can define 
$$
 M:=\inf_{v\in S^1}\mathcal{I}(v).  
$$
Now, consider a sequence $(v_n)$ such that $\|v_n\|=1$ and $M=\lim\mathcal{I}(v_n)$ . Going if necessary to a subsequence, we may assume that  $v_n \rightharpoonup v_0$ in $E$ with $\|v_0\|\leq 1$ and by Lemma \ref{compactness} 
	\begin{equation*}
		A(v_n)\rightarrow A(v_0)\geq 0\quad\text{and}\quad B(v_n)\rightarrow B(v_0)\geq 0.
	\end{equation*}
We claim that $v_0 \neq 0$. Indeed, suppose that $v_0=0$. By  Lemma~\ref{existence of a function r C1}, there exists a sequence $r(v_n)>0$ such that
 \begin{equation}\label{Soda1}
     \|v_n\|^p=r(v_n)^{q-p}A(v_n)-r(v_n)^{s-p}B(v_n).
  \end{equation}
 Using that  $\|v_n\|=1$, we get
		\begin{align*}
		1=r(v_n)^{q-p}A(v_n)-B(v_n)r(v_n)^{s-p}
		\leq r(v_n)^{q-p}A(v_n).
	\end{align*}
Since $q>p$ and $A(v_n)\rightarrow0$, we obtain   $r(v_n)\rightarrow \infty$. On the other hand, we have
	\begin{align*}
		\mathcal{I}(v_n)=\left(\frac{1}{s}-\frac{1}{q}\right)B(v_n)r(v_n)^s+\left(\frac{1}{p}-\frac{1}{q}\right)r(v_n)^p
		\geq \left(\frac{1}{p}-\frac{1}{q}\right)r(v_n)^p.
	\end{align*}
Taking the limit above, we obtain a contradiction and $v_0\neq 0$. From the last inequality, up to a subsequence, we can assume that $r(v_n)\rightarrow r_0\geq0$ and taking to the limit in \eqref{Soda1} we obtain 	\begin{equation}\label{Zico}
		r_0^{p-q}+B(v_0)r_0^{s-q}=A(v_0),
	\end{equation}
which implies that $r_0>0$.	

Next, we shall prove that  $\|v_0\|=1$. Otherwise, there exists $\mu>1$ such that $\|\mu v_0\|=1$. From Lemma~\ref{existence of a function r C1}, there are $r(v_0)>0$ such that 
$$
\|v_0\|^pr(v_0)^{p-q}+B(v_0)r(v_0)^{s-q}=A(v_0).
$$
This, combined with \eqref{Zico} and the fact that  $\mu>1$ implies 
$$
r_0^{p-q}+B(v_0)r_0^{s-q}<r(v_0)^{p-q}+B(v_0)r(v_0)^{s-q}, 
$$
equivalently 
$$
r(v_0)^{p-q}\left[\left(\frac{r(v_0)}{r_0}\right)^{q-p}-1\right]+B(v_0)r(v_0)^{s-q}\left[\left(\frac{r(v_0)}{r_0}\right)^{q-s}-1\right]<0.
$$
Since $\max\{s,p\}<q$, we have that $r_0>r(v_0)$. Now, consider the function 
\begin{equation*}
		\psi(t)=\left(\frac{1}{s}-\frac{1}{q}\right)B(v_0)t^s+\left(\frac{1}{p}-\frac{1}{q}\right)\|v_0\|^pt^p, \quad t>0
\end{equation*}
and observe that $\psi$ is strictly increasing. Thus,
\begin{align*}
M=\liminf_{n\rightarrow \infty}\mathcal{I}(v_n)\geq\left(\frac{1}{s}-\frac{1}{q}\right)B(v_0)r_0^s+\left(\frac{1}{p}-\frac{1}{q}\right)r_0^p\|v_0\|^p=&\psi(r_0).
\end{align*}
On the other hand, we have
\[
\psi(r_0)>\psi(r(v_0))=I(r(v_0)v_0)=I(\mu r(\mu v_0)v_0)=\mathcal{I}(\mu v_0),
\]
which contradicts the definition of $M$ because $\|\mu v_0\|=1$ and hence we concluded that $\|v_0\|=1$. From \eqref{Zico} and the uniqueness of the solution $r(v_0)$ we have $r_0=r(v_0)$ and
\begin{align*}
		M=\lim_{n \rightarrow \infty}\mathcal{I}(v_n)
		=& \lim_{n\rightarrow \infty} \left(\frac{1}{s}-\frac{1}{q}\right)B(v_n)r(v_n)^s+\left(\frac{1}{p}-\frac{1}{q}\right)r(v_n)^p\\
		=& \left(\frac{1}{s}-\frac{1}{q}\right)B(v_0)r_0^s+\left(\frac{1}{p}-\frac{1}{q}\right)r_0^p\\
		=& \mathcal{I}(v_0).
\end{align*}
Since $v_0$ is a critical point of $\mathcal{I}$ under $S^1$ so is $|v_0|$ and we can assume $v_0\geq0$. Applying Lemma \ref{fibering method } with $H(v)=\|v\|^p$, we conclude that  $u=r_0 v_0$ is a critical point of $I$, which completes the proof.
\end{proof}

\begin{proof}[Proof of Theorem \ref{Segundo Existencia}] For each fixed  $v \in E\backslash\{0\}$, by Lemma \ref{existence of a function r C1} there exist $r(v)$ such that 
 \begin{equation}\label{c condition 2} 
     \|v\|^pr(v)^{p-q}+B(v)r(v)^{s-q}=A(v),
 \end{equation}
and hence, we can write the reduced functional $\mathcal{I}$ as
	\begin{align*}
		\mathcal{I}(v)=& \left(\frac{1}{s}-\frac{1}{q}\right)B(v)r(v)^s+\left(\frac{1}{p}-\frac{1}{q}\right)\|v\|^pr(v)^p<0.
	\end{align*}
 
If $\|v\|=1$, we see that $A$ and $B$ are bounded because of our embedding results. From  \eqref{c condition 2}, it follows that
\begin{equation*}
    0<r(v)^{p-q}\leq r(v)^{p-q}+B(v)r(v)^{s-q}=A(v),
\end{equation*}
which implies that $r$ is bounded because $p>q$. Therefore, we can consider the minimization problem
	$$
		-\infty<M:=\inf_{v \in S^1}\mathcal{I}(v)<0.  
	$$
	 Let $(v_n) \subset S^1$ be a minimizing sequence. Up to a subsequence, we can assume that $v_n \rightharpoonup v_0$ weakly in $E$ with $\|v_0\|\leq1$. 
	Furthemore, by Lemma \ref{compactness} 
	\begin{equation*}
		A(v_n)\rightarrow A(v_0)\quad\text{and}\quad B(v_n)\rightarrow B(v_0).
	\end{equation*}
 Since $(r(v_n))$ is bounded, up to a subsequence, we can assume that $r(v_n)\rightarrow r_0\geq 0$.

 Now observe that $I$ is weakly lower semicontinuous and $r(v_n)v_n \rightharpoonup r_0v_0$, then
  \[
  I(r_0v_0)\leq \liminf I(r(v_n)v_n)=\liminf \mathcal{I}(v_n)=M<0
  \]
and so $r_0v_0\neq 0$. From \eqref{c condition 2}, we have 	
\begin{equation*}\label{r0 equation}
		\|v_n\|^pr(v_n)^{p-q}+B(v_n)r(v_n)^{s-q}= A(v_n).
	\end{equation*}
Passing to the limit and observing that $\|v_0\|\leq 1$ we obtain
\begin{equation*}
    \|v_0\|^p r_0^{p-q}+B(v_0)r_0^{s-q}\leq A(v_0).
\end{equation*}
On the other hand, applying Lemma~\ref{existence of a function r C1} for $v_0$, we have
\begin{equation*}\label{key equation}
    \|v_0\|^pr(v_0)^{p-q}+B(v_0)r(v_0)^{s-q}=A(v_0),
\end{equation*}
which implies that $r_0\leq r(v_0)$. Now, suppose by contradiction that $r_0<r(v_0)$ and consider the function 
	\begin{equation*}
		\psi(t):=I(tv_0)=\frac{t^p}{p}\|v_0\|^p-\frac{t^q}{q}A(v_0)+\frac{t^s}{s}B(v_0),\quad t\in [0,r(v_0)]
	\end{equation*}
	and observe that $\psi$ is strictly decreasing. Indeed, first note that $\psi(0)=0$ and $\psi(r(v_0))=\mathcal{I}(v_0)<0$. In addition, we observe that $\psi'(0)=0$ and for $t\neq 0$,
 \begin{equation*}
     0=\psi'(t)=t^{p-1}\|v_0\|^p-t^{q-1}A(v_0)+t^{s-1}B(v_0)\Leftrightarrow \|v_0\|^p=t^{q-p}A(v_0)-t^{s-p}B(v_0)\Leftrightarrow t=r(v_0).
 \end{equation*}
Consequently, $\psi$ must be strictly decreasing on $[0,r(v_0)]$. Thus,
 \begin{equation*}
M=\liminf I(r(v_n)v_n)\geq I(r_0v_0)>I(r(v_0)v_0)=\mathcal{I}(v_0).    
 \end{equation*}
By Lemma \ref{existence of a function r C1} we have $\mu r(\mu v )=r(v)$ for all $v\neq 0$ and taking $\mu =\|v_0\|^{-1}$ we have $\mu v_0 \in S^1$ and
\begin{equation*}
    \mathcal{I}(\mu v_0)=I(\mu r(\mu v_0)v_0)=I(r(v_0)v_0)=\mathcal{I}(v_0)< M,
\end{equation*}
which is a contradiction and therefore $r(v_0)=r_0$. Then,
\begin{equation*}
    1=\lim_{n\rightarrow \infty} \|v_n\|^p=r_0^{q-p}A(v_0)-r_0^{s-p}B(v_0)=\|v_0\|^p.
\end{equation*}
and
\begin{align*}
		M=\lim_{n \rightarrow \infty}\mathcal{I}(v_n)
		=& \lim_{n\rightarrow \infty} \left(\frac{1}{s}-\frac{1}{q}\right)B(v_n)r(v_n)^s+\left(\frac{1}{p}-\frac{1}{q}\right)\|v_n\|^pr(v_n)^p\\
		=& \left(\frac{1}{s}-\frac{1}{q}\right)B(v_0)r_0^s+\left(\frac{1}{p}-\frac{1}{q}\right)\|v_0\|^pr_0^p\\
		=& \mathcal{I}(v_0).
	\end{align*}
Since $v_0$ is a critical point of $\mathcal{I}$ under $S^1$, it follows that $|v_0|$ is also a critical point. Thus, we can assume without loss of generality that $v_0\geq 0$. Applying Lemma \ref{fibering method }, we conclude that  $u=r_0 v_0$ is a critical point of $I$, which completes the proof.
\end{proof}

Moving forward, we are proceeding to prove our third existence result.

\begin{lemma}  Assume the assumptions of Theorem \ref{Terceiro Exitencia}. Then for each $v \in E\backslash\{0\}$  the function $G(.,v)$ defined by \eqref{G definition} has a unique critical point at 	
		\begin{equation}\label{critical point r barra}
			\overline{r}(v)=\left(\frac{A(v)(q-p)}{B(v)(s-p)}\right)^{1/(s-q)}.
		\end{equation}
		Moreover, 
		\begin{equation}\label{G de r barra}
			G(\overline{r}(v),v)=	\max_{r>0}G(r,v)=\left(\frac{A(v)^{s-p}}{\eta(p,q,s) B(v)^{q-p}}\right)^{1/(s-q)}>0.
		\end{equation}
		\end{lemma}

Under the assumptions in Theorem \ref{Terceiro Exitencia}, we introduce the set,
$$
		\Omega:=\{v \in E\backslash\{0\}: \|v\|^p<G(\overline{r}(v),v)\}.
  $$
 
\begin{remark}\label{OmegaNaoVaZIO}
If $\mathcal{D}_1$ is the set defined in \eqref{D hypothesis}, then $\mathcal{D}_1\subset \Omega$ and hence $\Omega\neq\emptyset$. Indeed, first, we observe that by \eqref{critical point r barra} we have
	\begin{equation}\label{B equation}
		B(v)=\left(\frac{q-p}{s-p}\right)\overline{r}(v)^{q-s}A(v),
	\end{equation}
	and  from the definition of $G$, we obtain
	\begin{equation}\label{G(r,v) alternative form}
		G(\overline{r}(v),v)=\overline{r}(v)^{q-p}A(v)-\left(\frac{q-p}{s-p}\right)\overline{r}(v)^{q-p}A(v)= \left(\frac{s-q}{s-p}\right) A(v)\overline{r}(v)^{q-p}.
	\end{equation}	
	If $v \in \mathcal{D}_1$, we see that 
	\begin{align*}
		\|v\|^p<&\left(\frac{p}{q}\right)\eta(p,q,s)^{\frac{1}{q-s}}A(v)^{\frac{s-p}{s-q}}B(v)^{\frac{p-q}{s-q}}.
	\end{align*}
	Thus, from \eqref{B equation} and \eqref{G(r,v) alternative form}, it follows 
	\begin{align*}
		\|v\|^p&<\left(\frac{p}{q}\right)\frac{(s-q)(q-p)^{\frac{q-p}{s-q}}}{(s-p)^{\frac{s-p}{s-q}}}A(v)^{\frac{s-p}{s-q}}\left[\left(\frac{q-p}{s-p}\right)\overline{r}(v)^{q-s}A(v)\right]^{\frac{p-q}{s-q}}\\
		&=\frac{p}{q}\left(\frac{s-q}{s-p}\right)\overline{r}(v)^{q-p}A(v)\\
  &<G(\Bar{r}(v),v),
	\end{align*}
and so we conclude that $\mathcal{D}_1\subset \Omega$. 
\end{remark}

Next, we will prove some technical properties of $\Omega$ that play an essential role in proving Theorem~\ref{Terceiro Exitencia}.

\begin{lemma}\label{Propriedades Omega} If $p<q<s<p^*$, for each $v \in \Omega$ there exists a unique real number $r(v)>\overline{r}(v)$ such that the pair $(r(v),v)$ satisfies
$$
\|v\|^p=r(v)^{q-p}A(v)-r(v)^{s-p}B(v)=G(r(v),v),
$$
and $r\in C^1(\Omega, \mathds{R})$. Furthermore, for any $v\in \Omega$ and $\mu>0$, it holds $\mu v\in \Omega$, and as a consequence, $\Omega\cap S^1\neq\emptyset$.
\end{lemma}

\begin{proof} If $v\in\Omega$ we have $\|v\|^p<G(\bar{r}(v),v)$. Since $G(r,v)=r^{q-p}\left(A(v)-B(v)r^{s-q}\right)$ and $p<q<s$, it follows that     
$$
\lim_{r\rightarrow \infty}G(r,v)=-\infty,
$$
 and so by the Intermediate Value Theorem, there exists a real number $r(v)>\bar{r}(v)$ such that the pair $(r(v),v)$ verifies $\|v\|^p=G(r(v),v)$. To prove that $r(v)$ is unique, we observe $(q-p)\bar{r}(v)^{q-p}A(v)=(s-p)\bar{r}(v)^{s-p}B(v)$ and hence we can write 
$$
G(r,v)=A(v)\left(r^{q-p}-\frac{q-p}{s-p}\bar{r}(v)^{q-s}r^{s-p}\right).
$$ 
Consequently, 
$$
\begin{aligned}
\frac{\partial G}{\partial r}(r,v)&=(q-p)r^{s-p-1} A(v)(r^{q-s}-\overline{r}(v)^{q-s})\\
&=(q-p)r^{s-p-1} A(v)\left(\frac{1}{r^{s-q}}-\frac{1}{\overline{r}(v)^{s-q}}\right)<0,
\end{aligned}
$$
for all $r>\bar{r}(v)$, thereby implying the uniqueness of $r(v)$.

To verify that the map $r$ is a $C^1$, by setting $r=r(v)$, we obtain
	\begin{equation}\label{G_r <0}
		\frac{\partial G}{\partial r}(r(v),v)=(q-p) r(v)^{p-s+1}A(v)(r(v)^{q-s}-\overline{r}(v)^{q-s})<0.
	\end{equation}
Now, considering the function $f:(0,\infty)\times \Omega\rightarrow \mathds{R}$ given by 
	\begin{equation*}
		f(r,v)=G(r,v)-\|v\|^p,
	\end{equation*}
by \eqref{G_r <0} we see that $\frac{\partial f}{\partial r}(r(v),v)<0$. Thus, using the Implicit Function Theorem and arguing as in the proof of Lemma \ref{existence of a function r C1}, we obtain that $r \in C^1(\Omega,\mathds{R})$.

Now, suppose that $v\in\Omega$, that is 
$$
\|v\|^p<G(\bar{r}(v),v)=\frac{s-q}{s-p}\bar{r}(v)^{q-p}A(v).
$$
Since, $\bar{r}(v)=\mu\bar{r}(\mu v)$ for all $\mu>0$, we get  	
$$
\|\mu v\|^p<\frac{s-q}{s-p}(\bar{r}(\mu v))^{q-p}A(\mu v)=G(\bar{r}(\mu v),\mu v),
$$
which implies that $\mu v\in\Omega$ and this completes the proof.
\end{proof}.

\begin{lemma}\label{Binferior} The following statement holds
$$
\inf_{v \in \Omega\cap S^1}B(v)>0.
$$
\end{lemma}
\begin{proof}
For any $v \in \Omega\cap S^1$,
from \eqref{critical point r barra} we get
	\begin{align*}
		1=\|v\|^p<G(\overline{r}(v),v)=&\left(\frac{s-q}{s-p}\right)A(v)\overline{r}(v)^{q-p}\\
		=& \left(\frac{s-q}{s-p}\right)A(v)\left(\frac{A(v)(q-p)}{B(v)(s-p)}\right)^{(q-p)/(s-q)},
	\end{align*}
	which implies 
 $$
		\eta(p,q,s)B(v)^{q-p}<A(v)^{s-p}.
  $$
By using H\"{o}lder's inequality, it follows
	\begin{align*}
		A(v)=\int_{\mathds{R}^N_+}a|v|^q\, \mathrm{d} x
		=& \int_{\mathds{R}^N_+} \frac{a}{b^{q/s}}b^{q/s}|v|^q\, \mathrm{d} x\\
		\leq& \left(\int_{\mathds{R}^N_+} \left[\frac{a}{b^{q/s}}\right]^{\frac{s}{s-q}}\, \mathrm{d} x\right)^{(s-q)/s}\left(\int_{\mathds{R}^N_+} b|v|^s\, \mathrm{d} x\right)^{q/s}.
	\end{align*}
Thus,  we get
	$
		A(v)\leq C_{a,b}B(v)^{q/s}
	$
	with
	\[
	0<C_{a,b}=\left(\int_{\mathds{R}^N_+} \left[\frac{a^{1/q}}{b^{1/s}}\right]^{\frac{sq}{s-q}}\, \mathrm{d} x\right)^{(s-q)/s},
	\]
 which is finite due to assumption \eqref{ab condition}. By combining the above inequalities, we obtain
	\[
	\eta(p,q,s)B(v)^{q-p}<A(v)^{s-p}\leq C_{a,b}^{s-p}B(v)^{(s-p)q/s}, 
	\]
and hence 
	$
		0<\eta(p,q,s) C_{a,b}^{p-s}< B(v)^{(s-q)p/s}$,
thereby yielding the desired result.
\end{proof}

\begin{lemma}\label{Lira}If $\mathcal{D}_1$ is the set defined in \eqref{D hypothesis}, then  $\mathcal{D}_1\cap S^1\neq \varnothing$, where $S^1$ is the unit sphere in $E$. Moreover,
\begin{equation}\label{I Negativo em D}
\mathcal{I}(v)<0, \quad \forall v\in\mathcal{D}_1. 
\end{equation}
\end{lemma}
\begin{proof}
     If $v \in \mathcal{D}_1$, the computation in Remark~\ref{OmegaNaoVaZIO} shows that 
 \begin{equation}\label{Desigualadde D}
 \|v\|^p<\frac{p}{q}\left(\frac{s-q}{s-p}\right)\bar{r}(v)^{q-p}A(v),
 \end{equation}
 and for $\mu>0$, by \eqref{critical point r barra} we easily obtain 
$$
\mu \bar{r}(\mu v)=\bar{r}(v),\quad \forall v \in E\backslash\{0\}.
 $$
 Thus, 
 	\begin{align*}
		\|\mu v\|^p<&\frac{p}{q}\left(\frac{s-q}{s-p}\right)\mu^{p-q}\bar{r}(v)^{q-p}A(\mu v)
		=\frac{p}{q}\left(\frac{s-q}{s-p}\right)\bar{r}(\mu v)^{q-p}A(\mu v),
	\end{align*} 
which implies that $\mu v \in \mathcal{D}_1$. In particular, choosing $\mu=\|v\|^{-1}$ we conclude that $\mathcal{D}_1\cap S^1\neq \varnothing$.	

To verify \eqref{I Negativo em D}, since the pair $(r(v),v)$ satisfies 
\begin{equation*}
		B(v)r(v)^s=A(v)r(v)^q-\|v\|^pr(v)^p, 
	\end{equation*}	
 from  \eqref{reduced functional}, the fact that $\bar{r}(v)<r(v)$ for each $v \in \mathcal{D}_1\subset\Omega$ and inequality \eqref{Desigualadde D}, we get
	\begin{align*}
		\mathcal{I}(v)
		=& \left(\frac{1}{s}-\frac{1}{q}\right)A(v)r(v)^q+\left(\frac{1}{p}-\frac{1}{s}\right) \|v\|^p r(v)^p\\
		<& \left(\frac{1}{s}-\frac{1}{q}\right)A(v)r(v)^q+\left(\frac{1}{p}-\frac{1}{s}\right)\frac{p}{q}\left(\frac{s-q}{s-p}\right)A(v)r(v)^q.
	\end{align*}
Since the last term of the inequality above is zero, this completes the proof. \end{proof}

%%%%%%%%%%%%%%%%%%%%%%%%%%%%%%
\begin{proof}[Proof of Theorem \ref{Terceiro Exitencia}:] Let $r \in C^1(\Omega,\mathds{R})$ be the function given by the Lemma \ref{Propriedades Omega}. For $v \in S^1$, we have
	\begin{equation*}
		1=r(v)^{q-p}A(v)-r(v)^{s-p}B(v),
	\end{equation*}
	which implies that
	\begin{equation}\label{A/B}
		r(v)<\left(\frac{A(v)}{B(v)}\right)^{1/(s-q)}, \quad v \in S^1.
	\end{equation}
Since $A$ is bounded in $S^1$, by Lemma~\ref{Binferior} we have that $r$ is bounded in $\Omega\cap S^1$. Hence, $\mathcal{I}$ is lower bounded in $\Omega\cap S^1$ and hence in view of Lemma~\ref{Lira}
\begin{equation}\label{variational problem}
		M=\inf_{v \in \Omega\cap S^1} \mathcal{I}(v)<0.
\end{equation}

Let $(v_n)\subset \Omega\cap S^1$ be a minimizing sequence. Up to a subsequence, $v_n \rightharpoonup v_0$ weakly in $E$ with $\|v_0\|\leq1$. Lemmas~\ref{compactness} then imply
	$$A(v_n)\rightarrow A(v_0)\quad\mbox{and}\quad B(v_n)\rightarrow B(v_0).
 $$
By Lemma~\ref{Propriedades Omega}, the sequence $(r(v_n))$ satisfies $r(v_n)>\Bar{r}(v_n)$. Moreover, from \eqref{A/B}, the sequence $(r(v_n))$ is bounded and, up to a subsequence, we can assume that $r(v_n)\rightarrow r_0\geq0$. Thus, we obtain 
\[
	0>M=\liminf\mathcal{I}(v_n)\geq\left(\frac{1}{p}-\frac{1}{q}\right)A(v_0) r_0^q+\left(\frac{1}{s}-\frac{1}{p}\right)B(v_0)r_0^s.
	\]
Considering that $p<q<s$, we concluded that $r_0>0$. Furthermore,  from \eqref{critical point r barra}, it follows 
	\begin{equation*}\label{r(vn) convergence}
	\begin{aligned}
\lim_{n\rightarrow+\infty}\overline{r}(v_n)=\lim_{n\rightarrow+\infty}\left(\frac{A(v_n)(q-p)}{B(v_n)(s-p)}\right)^{1/(s-q)}&=\left(\frac{A(v_0)(q-p)}{B(v_0)(s-p)}\right)^{1/(s-q)}=\overline{r}(v_0),
\end{aligned}
	\end{equation*}
and hence $r_0\geq \bar{r}(v_0)$. Furthermore, 
	\begin{equation*}\label{G(r,vn)}
 \begin{aligned}
     	\lim_{n\rightarrow+\infty} G(\overline{r}(v_n),v_n)=G(\overline{r}(v_0), v_0).
\end{aligned}
\end{equation*}
Since $v_n \in \Omega$, we get 
	\begin{equation*}
		\|v_0\|^p\leq \liminf_{n \rightarrow \infty} \|v_n\|^p \leq \liminf_{n \rightarrow \infty} G(\overline{r}(v_n),v_n)=G(\overline{r}(v_0),v_0).
	\end{equation*}
Assume by contradiction that $v_0\not\in \Omega$, that is, 
	$
		\|v_0\|^p=G(\overline{r}(v_0),v_0)$. Since $\|v_n\|^p=G(r(v_n),v_n)$, taking to the limit we get
$$	
G(\overline{r}(v_0),v_0)=\|v_0\|^p\leq \liminf_{n \rightarrow \infty} \|v_n\|^p=\liminf_{n \rightarrow \infty}G(r(v_n),v_n)=G(r_0,v_0),
$$
which implies that $\overline{r}(v_0)=r_0$ because $\overline{r}(v_0)$ is the global maximum of $G(.,v_0)$. Then, $r(v_n)\rightarrow \overline{r}(v_0)$ and from the definition of $\mathcal{I}$ and \eqref{critical point r barra} we obtain
\begin{align*}
		M=\lim_{n \rightarrow \infty}\mathcal{I}(v_n)&= \left(\frac{1}{p}-\frac{1}{q}\right)A(v_0)\overline{r}(v_0)^q+\left(\frac{1}{s}-\frac{1}{p}\right)B(v_0)\overline{r}(v_0)^s
	\end{align*}
and $(s-p)B(v_0)\bar{r}(v_0)^s=A(v_0)(q-p)\bar{r}(v_0)^q$. As a consequence, we infer that 
\begin{align*}
		M&= A(v_0)\overline{r}(v_0)^q\left[\left(\frac{1}{p}-\frac{1}{q}\right)+\left(\frac{1}{s}-\frac{1}{p}\right)\frac{(q-p)}{s-p}\right]\\
  &=A(v_0)\overline{r}(v_0)^q\frac{(q-p)}{p}\left(\frac{1}{q}-\frac{1}{s}\right)>0
	\end{align*}
because $p<q<s$, which contradicts \eqref{variational problem} and hence we conclude that $v_0 \in \Omega$. 

\noindent\textit{Claim}: $r_0=r(v_0)$.

Assuming that the claim is valid, we can take the limit at 
$$ 
 1=\|v_n\|^p=G(r(v_n),v_n),
 $$
 to obtain
\begin{equation*}
1=A(v_0)r(v_0)^{q-p}-B(v_0)r(v_0)^{s-p}=G(r(v_0),v_0)=\|v_0\|^p.
\end{equation*}
Thus, we conclude that $v_0\in \Omega\cap S^1$ and we also have 
$$
M=\lim_{n\rightarrow\infty} \mathcal{I}(v_n)=\left(\frac{1}{p}-\frac{1}{q}\right)A(v_0)r(v_0)^q+\left(\frac{1}{s}-\frac{1}{p}\right)B(v_0)r(v_0)^s=\mathcal{I}(v_0).
$$
Therefore, by Lemma \ref{fibering method }, $r(v_0)v_0$ is a nonnegative and nontrivial critical point if $I$ in $E$. This completes the proof of Theorem~\ref{Terceiro Exitencia}.

It remains to prove $r_0=r(v_0)$. Since $v_0 \in \Omega$, by Lemma~\ref{Propriedades Omega}, we can choose $\mu_0>0$ such that $\mu_0 v_0 \in \Omega\cap S^1$. By Lemma~\ref{Propriedades Omega} we know that $r( v_0)>\Bar{r}( v_0)$ and $G(r(v_0),v_0)=\|v_0\|^p$. Taking the limit at $\|v_n\|^p=G(r(v_n),v_n)$, we get 
$
\|v_0\|^p\leq G(r_0,v_0)$. Consequently,
$$
G(r(v_0),v_0)=\|v_0\|^p\leq G(r_0,v_0).
$$
Since $G(r,v_0)$ is decreasing for $r\geq \bar{r}(v_0)$ and $r(v_0)>\bar{r}(v_0)$, it follows that $r_0\leq r(v_0)$. We have 
$$
\bar{r}(v_0)\leq r_0\leq r(v_0).
$$
Suppose by contradiction that $r_0<r(v_0)$. Since $G(r,v_0)$ is strictly decreasing for all $r\in (r_0, r(v_0))$, we see that 
\begin{equation*}
		\|v_0\|^p=G(r(v_0),v_0)< G(r,v_0), \quad \forall r\in[r_0,r(v_0)).
	\end{equation*} 
Considering the function
$$
h(r)= I(rv_0), \quad r \in (r_0, r(v_0)),
$$
 a straightforward computation shows that 
 $$
 h'(r)=r^{p-1}\left(\|v_0\|^p-G(r,v_0)\right)<0,
 $$
 which implies that $h$ is strictly decreasing. Thus, we get 
	\begin{align*}
		M=\liminf_{n \rightarrow \infty}I(r(v_n)v_n)
		\geq I(r_0v_0)
		> I(r(v_0)v_0)
		&=I(r(\mu_0 v_0)\mu_0 v_0)= \mathcal{I}(\mu_0 v_0),
	\end{align*}
with $\mu_0 v_0\in \Omega\cap S^1$. This contradicts the definition of $M$ and hence $r_0=r(v_0)$. 
\end{proof}

%%%%%%%%%%%%%%%%%%%%%%%%%%%%%%%%%%%%%%%%%%%%%%%%%%%%%%%%%%%%%%%%%%%%%%%%%%%%%%%%%%%%%%%%%%%%%%%%%%%%%%%%%%%%%%%%%%%%%%%%%

\section{Final comments}
In this section, we explore potential future developments stemming from the results established in this work.

\begin{itemize}
\item Our results demonstrate the existence of solutions for problems involving nonlinearities within the subcritical growth range in terms of new Sobolev embedding proved in the present work. 
It would be interesting to explore the existence of solutions for nonlinearities with corresponding critical growth.
\item We address the case $p=N$ with polynomial growth. It is also important to consider scenarios
where the nonlinearities exhibit exponential growth in the fashion of Trudinger-Moser-type inequalities.
\item We observe that inequality \eqref{hardy with xn} can be used to obtain Sobolev trace type inequality, which allows one to treat problems with nonlinear boundary conditions as in the work \cite{zbMATH01215632}.
\item We observe that the condition $\gamma > p-1$ is sufficient to prove the Hardy type inequality \eqref{hardy with xn}, which aligns with the sufficient condition proven in  \cite{zbMATH05377123}. However, determining a necessary condition on $\gamma$ for \eqref{hardy with xn} remains an open question.
\item If $\gamma\geq p$, we observe that  $\mathcal{D}_\gamma^{1,p}(\mathds{R}^N_+)$ embeds into $ W^{1,p}(\mathds{R}^N_+)$, resulting in bounded state solutions.Therefore, a natural question arises: what happens when 
$p-1<\gamma<p$?
\end{itemize}

%%%%%%%%%%%%%%%%%%%%%%%%%%%%%%%%%%%%%%%%%%%%%%%%%%%%%%%%%%%%%%%%%%%%%%%%%%%%%%%%%%%%%%%%%%%%%%%%%%%%%%%%%%%%%%%%%%%%%%%%%%%%%%%%%%%%%%%%%%% %

%%%%%%%%%%%%%%%%%%%%%%%%%%%%%%%%%%%%%%%%%%%%%%%%%%%%%%%%%%%%%%%%%%%%%%%%%%%%%%%%%%%%%%%%%%%%%%%%%%
	%  Statements and declarations
%%%%%%%%%%%%%%%%%%%%%%%%%%%%%%%%%%%%%%%%%%%%%%%%%%%%%%%%%%%%%%%%%%%%%%%%%%%%%%%%%%%%%%%%%%%%%%%%%%

\begin{flushleft}
	{\bf Funding:}  
	J. M. do \'O acknowledges partial support from CNPq through grants 312340/2021-4, 409764/2023-0, 443594/2023-6, CAPES MATH AMSUD grant 88887.878894/2023-00 and E. Medeiros acknowledges partial support from CNPq through grant 310885/2023-0
	and Para\'iba State Research Foundation (FAPESQ), grant no 3034/2021.  \\
	{\bf Ethical Approval:}  Not applicable.\\
	{\bf Competing interests:}  Not applicable. \\
	{\bf Authors' contributions:}    All authors contributed to the study conception and design. All authors performed material preparation, data collection, and analysis. The authors read and approved the final manuscript.\\
	{\bf Availability of data and material:}  Not applicable.\\
	{\bf Ethical Approval:}  All data generated or analyzed during this study are included in this article.\\
	{\bf Consent to participate:}  All authors consent to participate in this work.\\
	{\bf Conflict of interest:} The authors declare no conflict of interest. \\
	{\bf Consent for publication:}  All authors consent for publication. \\
\end{flushleft}

\bigskip

%%%%%%%%%%%%%%%%%%%%%%%%%%%%%%%%%%%%%%%%%%%%%%%%%%%%%%%%%%%%%%%%%%%%%%%%%%%%%%%%%%%%%%%%%%%%%%%%%%
	%  REFERENCES
%%%%%%%%%%%%%%%%%%%%%%%%%%%%%%%%%%%%%%%%%%%%%%%%%%%%%%%%%%%%%%%%%%%%%%%%%%%%%%%%%%%%%%%%%%%%%%%%%%

 \bibliography{bibliography.bib}
 \bibliographystyle{abbrv}

\end{document}